\def \HAL{1} 

\author{Vincent Andrieu\thanks{V.~Andrieu is with Universit\'e Lyon 1 CNRS UMR 5007 LAGEP, France and Fachbereich C - Mathematik und Naturwissenschaften, Bergische Universit\"at Wuppertal, Germany.
\hfill \null \break
{\tt\small vincent.andrieu@gmail.com}
}, Bayu Jayawardhana\thanks{B.~Jayawardhana is with ENTEG, Faculty of Mathematics and
Natural Sciences, University of Groningen, the
Netherlands,\hfill \null \break
{\tt\small bayujw@ieee.org, b.jayawardhana@rug.nl}}, Laurent
	Praly\thanks{L.~Praly is with
MINES ParisTech,
PSL Research University,
Centre for automatic control and systems theory,
France,\hfill \null \break{\tt\small Laurent.Praly@mines-paristech.fr}}
}

\if\HAL 1
\documentclass{article}
\usepackage{amsthm} 
\usepackage[a4paper, total={6in, 8in}]{geometry}
\else
\documentclass[journal]{IEEEtran}
\IEEEoverridecommandlockouts                              
\fi

\usepackage{amsfonts,latexsym,amsmath,amssymb}
\usepackage[mathscr]{eucal}
\usepackage{graphicx}
\usepackage{comment}
\usepackage{hyperref}
\usepackage[usenames,dvipsnames]{xcolor}
\usepackage{subfigure}

\usepackage[normalem]{ulem}

\newtheorem{lemma}{Lemma}
\newtheorem{proposition}{Proposition}

\newtheorem{definition}{Definition}

\def\ddt{\frac{\mbox{\small d}\hphantom{\hskip 0.07em t}\null }{\mbox{\small
d}\hskip 0.07em t}}

\catcode`\@=11
\def\downparenfill{$\m@th\braceld\leaders\vrule\hfill\bracerd$}
\def\overparen#1{\mathop{\vbox{\ialign{##\crcr\crcr
\noalign{\kern0.4ex}
\downparenfill\crcr\noalign{\kern0.4ex\nointerlineskip}
$\hfil\displaystyle{#1}\hfil$\crcr}}}\limits}
\catcode`\@=12

\def\NN{{\mathbb N}}    
\def\RR{{\mathbb R}}    

\def\DR{\mathcal{D}}   

\def\de{{\widetilde e}}
\def\dE{{\widetilde E}}
\def\dG{{\widetilde G}}
\def\dk{{\widetilde k}}

\def\dr{{\widetilde r}}
\def\dx{{\widetilde x}}     
\def\dX{{\widetilde X}}     

\def\PR{{P}}

\def \kell{{\ell}}      
\def\der{{\mathfrak{d}}}

\def\bida{A}
\def\bidb{B}
\def\bidc{C}

\def\indice{k}

\def\startmodifOLD{\begingroup\color{black}} 
\def\stopmodifOLD{\endgroup}
\def\startmodif{\begingroup\color{black}} 
\def\stopmodif{\endgroup}
\def\startmodifbj{\begingroup\color{black}} 
\def\stopmodifbj{\endgroup}

\begin{document}

\title{
Transverse exponential stability and applications 
}
\maketitle
\begin{abstract}
We
investigate how the following properties are related to each other:
i) - A manifold is ``transversally'' exponentially stable; ii) -  The
``transverse'' linearization along any solution in the manifold is
exponentially stable; iii) -  There exists a
field of positive definite quadratic forms whose 
restrictions to the directions transversal to the 
manifold are decreasing along the flow.
\startmodifbj
We illustrate their relevance with the 
study of exponential incremental stability. Finally, we apply 
these results to two control design problems, nonlinear 
observer design and synchronization. In particular, we provide 
necessary and sufficient conditions for the design of nonlinear observer 
and of nonlinear synchronizer with exponential convergence property. 
\stopmodifbj
 \end{abstract}

\section{Introduction}

The property of attractiveness of a (non-trivial) invariant manifold
 is often sought in many control design  problems.   In the classical
internal model based output regulation \cite{IsidoriByrnes_TAC_90OutputReg}, it is known
that the closed-loop system must have an attractive invariant
manifold, on which, the tracking error is equal to zero.  In the
Immersion \& Invariance \cite{AstolfiOrtega_TAC_03_InvNewToolStabAdapCont}, in the sliding-mode
control approaches, or observer designs  
\cite{AndrieuPraly_SIAM_06},  obtaining an attractive manifold is an integral
part of the design procedure.  Many multi-agent system problems, such
as, formation control, consensus and synchronization problems, are 
also
closely related to the analysis and design of an attractive invariant
manifold, see, for example, \cite{DePersisJayawardhana_TCNS_14,ScardoviSepulchre_Aut_09SynchNetw,WielandEtAl_Aut_11IntModSyn}.


\startmodif The study of stability and/or attractiveness of invariant
manifolds and more generally of sets has a long history.  See for
instance \cite[\S 16]{Yoshizawa_Book_1966stability} and the references therein.  In
this paper, we \stopmodif \startmodifbj focus on the \stopmodifbj 
\startmodif exponential convergence \stopmodif \startmodifbj property 
by studying \stopmodifbj the  system linearized transversally to the manifold.
We show that this
attractiveness property is equivalent to the existence of positive
definite quadratic forms which are decreasing along the flow of the
transversally linear system.  For constant quadratic forms and when
the system has some specific structure, the latter becomes the
Demidovich criterion which is a sufficient, but not yet necessary,
condition for convergent \startmodif systems \stopmodif
\cite{PavlovVanWouhNijmeijer_SCL_04ConvSystTribDemidovich,Pavlov_Book_05,RufferEtAl_SCL_13ConvSystIncStab}.
On the other hand, if we consider the standard output regulation
theory as pursued in \cite{IsidoriByrnes_TAC_90OutputReg}, the
attractiveness of the invariant manifold is established using the
center manifold theorem which corresponds to the stability property of
the linearized system at an equilibrium point.  Due to the lack of
characterization of an attractive invariant manifold, most of the
literature on constructive design for nonlinear output regulator is
based on various different sufficient conditions that can be very
conservative.  In these regards, our main results can potentially
provide a new framework for control designs aiming at making an
invariant manifold attractive.

The paper is divided into two parts.
\startmodifbj
In Subsection \ref{mainresult}, 
we study 
\stopmodifbj 
a dynamical system that admits a transverse exponentially stable
invariant manifold.  In particular, we establish  equivalent relations between:
\begin{list}{}{%
\parskip 0pt plus 0pt minus 0pt%
\topsep 0pt plus 0pt minus 0pt
\parsep 0pt plus 0pt minus 0pt%
\partopsep 0pt plus 0pt minus 0pt%
\itemsep 0pt plus 0pt minus 0pt
\settowidth{\labelwidth}{(iii)}%
\setlength{\labelsep}{0.2em}%
\setlength{\leftmargin}{\labelwidth}%
\addtolength{\leftmargin}{\labelsep}%
}
\item[(i)]
the transverse exponential stability of an invariant manifold; 
\item[(ii)]
the exponential stability of the transverse linearized system; 
\item[(iii)]
the existence of 
field of positive definite quadratic forms
the restrictions to the transverse direction to the
manifold of which are decreasing along the flow.
\end{list}
\startmodifbj
We illustrate these results by 
considering a particular case of exponential incremental stable systems in Subsection \ref{Sec_IncStability}. Here, incremental stability refers to the property where the distance between any two trajectories converges to zero (see, for example, \cite{Angeli_TAC_02_lyapIncStab,FromionScorletti_CDC_05,Angeli_TAC_09_FurthReslyapIncStab}).
For such systems, the property (i) $\Leftrightarrow$ (iii) is used to prove that the
exponential incremental stability property is equivalent to the
existence of a Riemannian distance which is contracted by the flow.\stopmodifbj 

\startmodifbj In the second part of the paper, we apply the
equivalence results to two different control problems: nonlinear
observer design and synchronization of nonlinear multi-agent systems.
In both problems, a necessary condition is obtained.  Based on this necessary
condition, we propose a novel design
for an observer, in Subsection \ref{Sec_Observer}, and for 
a synchronizer, in Subsection \ref{Sec_Synchronization}.


In Subsection \ref{Sec_Observer}, we reinterpret the three properties 
(i), (ii) and (iii) in the context of observer design. This allows us to revisit some of the results obtained recently in \cite{SanfelicePraly_TAC_12} and \cite{AndrieuBesanconSerres_CDC_13_NecObs}
and, more importantly, to show that the sufficient condition given in \cite{SanfelicePraly_TAC_12} is 
actually also a necessary condition to design an exponential (local) full-order observer.

Finally, in Subsection \ref{Sec_Synchronization}, we solve a 
nonlinear synchronization problem. In particular, we give some necessary and sufficient
conditions to achieve (local) exponential synchronization of nonlinear multi-agent systems involving more than two 
agents. This result generalizes our preliminary work in \cite{AndrieuJayawardhanaPraly_CDC_13}. Moreover, 
 under an extra assumption, we show how to obtain a global 
 synchronization for the two agents case.  

It is worth noting that our main results are applicable to other control problems beyond the two control problems mentioned before. In a recent paper by Wang, Ortega \& Su \cite{Wang_CDC_14}, our results have been applied to solve an adaptive control problem via the Immersion \& Invariance principle. \stopmodifbj

\startmodif
\section{Main result}
\label{sec_NormExpManifold}
\subsection{Transversally exponentially stable manifold}
\label{mainresult}
\stopmodif
Throughout this section, we consider a system in the form
\begin{equation}
\label{eq_System}
\dot e = F(e,x)\ ,\quad\dot x = G(e,x)
\end{equation}
where $e$ is in $\RR^{n_e}$, $x$ is in $\RR^{n_x}$
and the functions
$F:\RR^{n_e}\times\RR^{n_x}\rightarrow \RR^{n_e}$ and
$G:\RR^{n_e}\times\RR^{n_x}\rightarrow \RR^{n_x}$ are
$C^2$. We denote by $(E(e_0,x_0,t),X(
e_0,x_0,t))$ the (unique)
solution which goes through $(e_0,x_0)$ in $\RR^{n_e}\times\RR^{n_x}$
at time $t=0$. We assume it is defined for all positive times, i.e. the
system is {\it forward complete}.

\startmodif
The system (\ref{eq_System}) above can be used, for example, to study 
the behavior of two 
distinct solutions $X(x_1,t)$ and $X(x_2,t)$ of the system defined on $\RR^n$ by
\begin{equation}
\label{Aut_Syst}
\dot x = f(x) 
\end{equation}
Indeed,  we 
obtain an $(e,x)$-system of the type (\ref{eq_System}) with
\begin{equation}
\label{LP36}
F(e,x) = f(x+e) - f(x)\  ,\quad G(e,x)= f(x)\ .
\end{equation}
This is the context of incremental stability that we will use throughout this section to illustrate our main results.
\stopmodif

In the following, to simplify our notations, we denote by
$B_e(a)$ the open ball of radius $a$ centered at the origin in $\RR^{n_e}$.

We study the links between the following three  notions. 
\begin{list}{}{%
\parskip 0pt plus 0pt minus 0pt%
\topsep 1ex plus 0pt minus 0pt%
\parsep 0.5ex plus 0pt minus 0pt%
\partopsep 0pt plus 0pt minus 0pt%
\itemsep 1ex plus 0pt minus 0pt
\settowidth{\labelwidth}{1em}%
\setlength{\labelsep}{0.5em}%
\setlength{\leftmargin}{\labelwidth}%
\addtolength{\leftmargin}{\labelsep}%
}
\item[\textsf{
TULES-NL}]
\itshape
(Transversal uniform
local exponential stability)
\\
The system (\ref{eq_System}) is forward complete and there exist strictly positive real numbers $r$, $k$ and
$\lambda$ such
that we have, for all $(e_0,x_0,t)$ in $B_e(r)\times\RR^{n_x}\times
\RR_{\geq 0}$,
\begin{equation}
\label{eq_ExpStab}
|E(e_0,x_0,t)| \leq k |e_0| \exp(-\lambda t)
\  .
\end{equation}
\upshape
\item[\textsf{UES-TL}]
\itshape
(Uniform exponential stability for the transversally linear system)\\
The system
\begin{equation}
\label{LP4}
\dot{\dx }= \dG(\dx)  :=  G(0,\dx)
\end{equation}
is forward complete and there exist strictly positive real numbers $\dk $  and $\tilde \lambda$
such that any solution $(\dE(\de_0,\dx_0,t),\dX(\dx_0,t))$  of the transversally linear system
\begin{equation}
\label{eq_System_dif}
\dot \de = \frac{\partial F}{\partial e}(0,\dx)\de\ ,\quad\dot \dx = \dG(\dx)
\end{equation}
satisfies, for all $(\de_0,\dx_0,t)$ in $\RR^{n_e}\times\RR^{n_x}\times
\RR_{\geq 0}$,
\begin{equation}\label{eq_ExpStabDrift}
|\dE(\de_0,\dx,t)|\leq  \dk \exp(-\tilde \lambda t)|\de_0|
\ .
\end{equation}
\upshape
\item[\textsf{ULMTE}]
\itshape
(Uniform Lyapunov matrix transversal equation)
\\
For all positive definite matrix $Q$, there exists a continuous
function $\PR :\RR^{n_x}\rightarrow\RR^{n_e\times n_e}$
and strictly positive real numbers $\underline{p}$ and $\overline{p}$
such that $\PR $ has a derivative $\der_{\dG}\PR$ along $\dG$ in the
following sense
\begin{equation}
\label{LP9}
\der_{\dG} \PR(\dx)\; := \;
 \lim_{h\to 0}
\frac{\PR(\dX(\dx,h))-\PR(\dx)}{h}
\end{equation}
and we have, for all $\dx$ in $\RR^{n_x}$,
\begin{eqnarray}
\label{eq_TensorDerivative}
&\displaystyle
\hskip -3em
\der_{\dG} \PR(\dx) +
\PR (\dx ) \frac{\partial F}{\partial e}(0,\dx)
+ \frac{\partial F}{\partial e}(0,\dx)^\prime \PR(\dx)
\leq -Q
\\[0.5em]
\label{LP7}
&\displaystyle
\underline{p} \,  I \leq  \PR (\dx)\leq \overline{p}\,  I
\  .
\end{eqnarray}
\upshape
\end{list}

In other words, the system (\ref{eq_System}) is said to be TULES-NL if the manifold $\mathcal{E}:=\{(e,x):\,  e=0\}$ is
exponentially stable for the system
(\ref{eq_System}), locally in $e$
and uniformly in $x$; and it is said to be UES-TL if the manifold $\tilde{\mathcal{E}}:=\{(\dx,\de):\,  \de=0\}$ of the linearized system transversal to $\mathcal E$ in (\ref{eq_System_dif}) is exponentially stable uniformly in $\dx$.
 
\startmodif
Concerning the ULMTE property, condition (\ref{eq_TensorDerivative})
is related to the notion of horizontal contraction introduced in 
\cite[Section VII]{ForniSepculchre_TAC_2014}). 
However a key 
difference is that we do not require the monotonicity condition (\ref{eq_TensorDerivative}) to 
hold in the whole manifold $\RR^{n_e}\times\RR^{n_x}$ but only along the invariant submanifold  $\mathcal E$.
In this case the corresponding horizontal Finsler-Lyapunov function 
$V:\left(\RR^{n_x}\times\RR^{n_e}\right)\times\left(\RR^{n_x}\times\RR^{n_e}\right)$
that we get takes the form
$
V((x,e),(\delta_x,\delta _e)) =\delta _e ^\prime P(x) \delta _e
$.

In the case where the manifold is reduced to a single point, i.e. 
when the
system (\ref{eq_System}) is simply $\dot e=F(e)$ with an equilibrium point at the origin (i.e. $F(0)=0$) then 
\begin{itemize}
\item
the
TULES-NL property can be understood as the local exponential stability of the origin;
\item
the
UES-TL notion translates to the exponential stability of the
linear system
$\dot \de = \frac{\partial F}{\partial e}(0)\de$; and
\item
the  ULMTE concept is about the existence of a
positive definite matrix $P$ solution to the Lyapunov equation
$P\frac{\partial F}{\partial e}(0) + \frac{\partial F}{\partial
e}(0)^\top P = -Q$ where
$Q$ is an arbitrary positive definite matrix.
\end{itemize}%
In this particular case it is well known that these three properties 
are equivalent.
\stopmodif

\startmodifbj
For the example of incremental stability, as mentioned before, the three properties of TULES-NL, UES-TL and ULMTE can be understood globally as follows~:
\begin{list}{}{%
\parskip 0pt plus 0pt minus 0pt%
\topsep 1ex plus 0pt minus 0pt%
\parsep 0.5ex plus 0pt minus 0pt%
\partopsep 0pt plus 0pt minus 0pt%
\itemsep 1ex plus 0pt minus 0pt
\settowidth{\labelwidth}{1em}%
\setlength{\labelsep}{0.5em}%
\setlength{\leftmargin}{\labelwidth}%
\addtolength{\leftmargin}{\labelsep}%
}
\item[\textsf{P1 (TULES-NL)}]
\itshape
System (\ref{Aut_Syst}) is
globally
exponentially incrementally stable.
Namely there exist two strictly positive real numbers $k$ and $\lambda$ such that for all $(x_1,x_2)$ in $\RR^{n}\times\RR^n$ we have,
for all $t$ in $\RR_{\geq 0}$,
\begin{equation}\label{eq_ExpIncStab}
|X(x_1,t) - X(x_2,t)|\leq k|x_1-x_2|\exp(-\lambda t).
\end{equation}
\upshape
\item[\textsf{P2 (UES-TL) }]
\itshape
The manifold  $\mathcal{E}=\{(x,e),e=0\}$ is globally
exponentially stable for the system
\begin{equation}\label{DriftSystIncStab}
\dot e = \frac{\partial f}{\partial x}(x)e\ ,\quad\dot x = f(x)
\end{equation}
Namely there exist two strictly positive real numbers $k_e$ and $\lambda _e$
such that for all $(e,x)$ in $\RR^{n}\times\RR^n$, the corresponding
solution of (\ref{DriftSystIncStab}) satisfies
$$
|E(e,x,t) |\leq k_e |e|\exp(-\lambda _et)\ ,\
\forall t\in \RR_{\geq 0}
\  .
$$
\upshape
\item[\textsf{P3 (ULMTE)}]
\itshape
There exists a positive definite matrix $Q$
 in $\RR^{n\times n}$, a $C^2$ function $P
 :\RR^n\rightarrow\RR^{n\times n}$ and strictly positive real numbers
 $\underline{p}$ and $\overline{p}$ such that
$P $ has a derivative $\der_{f}P$ along $f$ in the
sense of (\ref{LP11}) ,and satisfies (\ref{eq_TensorDerivIncremental}) and(\ref{LP10}).
\upshape
\end{list}

\stopmodifbj \startmodif
In this context it is known that
P3 $\Rightarrow$ P1. Actually
asymptotic incremental stability for which Property P1 is a particular
case is known to be equivalent to the existence of an appropriate
Lyapunov function. This has been established in
\cite{Yoshizawa_ARMA_64_ExStaPerSol,TeelPraly_ESAIM_00,Angeli_TAC_02_lyapIncStab}
or \cite{RufferEtAl_SCL_13ConvSystIncStab}
for instance.
In our context, this Lyapunov function is given as a Riemannian distance. 
We shall show below that, as for the case of an equilibrium point, we have also P1 $\Rightarrow$ P2 
$\Rightarrow$ P3, (see Proposition \ref{theo_IncStab}), namely
incremental exponential stability implies the existence of
a Riemannian distance 
for which the flow is contracting.
\stopmodif

In studying the equivalence relation between 
\startmodif
TULES-NL,UES-TL  and  ULMTE,
\stopmodif
we are not interested in the possibility of a solution near the
invariant manifold to inherit some properties of solutions in this
manifold, such as, the asymptotic phase,
the shadowing property, 
the
reduction principle, etc.,
nor in the existence of some special coordinates allowing us to exhibit some invariant splitting in the dynamics (exponential dichotomy). 
This is the reason that,  besides forward
completeness,  we assume nothing for the  in-manifold dynamics
given by~:
$$ 
\dot{\tilde x} = \dG(\tilde x) = G(0,\tilde x)
\  . 
$$
So, 
for not misleading our reader,
we prefer to use the word ``transversal'' instead of ``normal'' as seen for instance in the various definitions of normally hyperbolic submanifolds given in \cite[\S 1]{HirschPughShub_Book_70invariant}.

 In order to simplify the exposition of our results  and to  concentrate our attention on
the main ideas, we assume everything is global and/or uniform, including restrictive
bounds. Most of this can be relaxed with working on open or compact
sets, but then with restricting the results to time intervals where a
solution remains in such a particular set.


\subsubsection{%
\startmodif
\textsf{
TULES-NL} ``$ \Rightarrow $'' \textsf{UES-TL}
\stopmodif
}
In the spirit of Lyapunov first method, we have  the following 
result. 

\begin{proposition}
\label{Prop_DetecNec}
 If Property \textsf{TULES-NL} holds and there exist positive
real numbers $\rho $,
$\mu$ and $c$ such that, for all $x$ in
$\RR^{n_x}$,
\begin{equation}
\label{LP1}
 \left|\frac{\partial F}{\partial e}(0,x)\right|\leq \mu
\ ,\quad
\left|\frac{\partial G}{\partial x}(0,x)\right|\leq \rho
\end{equation}
and, for all $(e,x)$ in $B_e(kr)\times\RR^{n_x}$,
\begin{equation}
\label{LP2}
\left|\frac{\partial^2 F}{\partial e\partial e}(e,x)\right|\leq c\ ,\
\left|\frac{\partial^2 F}{\partial x\partial e}(e,x)\right|\leq c\  ,\
\left|\frac{\partial G}{\partial e}(e,x)\right|\leq c
\; ,
\end{equation}
then Property \textsf{UES-TL} holds.
\end{proposition}
\vspace{0.2cm}

The proof of this proposition is  given  in Appendix \ref{SecProof_Prop_DetecNec}.
\startmodifbj Roughly speaking, it is \stopmodifbj 
based on the comparison between a given $e$-component of a solution
$\dE(\de_0,\dx_0,t)$ of (\ref{eq_System_dif}) with pieces of 
$e$-component of solutions   $E(\de_i,\dx_i,t-t_i)$ of solutions of
(\ref{eq_System}) where $\de_i,\dx_i$ are sequences of points defined on $\dE(\de_0,\dx_0,t)$.  
Thanks to  the bounds (\ref{LP1}) and (\ref{LP2}),
 it is possible to show that $\dE$ and $E$ remain sufficiently closed so that $\dE$ inherit the
convergence property of the solution $E$.
As a consequence, in
 the particular case in which $F$ does not depend on $x$, the two functions $E$ and $\dE$ 
do not depend on $x$ either and the bounds on the derivatives of the $G$ function 
are useless.

\subsubsection{%
\startmodif
\textsf{UES-TL}  ``$ \Rightarrow $'' \textsf{ULMTE}
\stopmodif
}
Analogous to the 
property of existence of a solution to the
Lyapunov matrix equation, we have the following proposition on the link between UES-TL and ULMTE notions. 
\begin{proposition}
\label{Prop_ExistTensor}
 If  
Property \textsf{UES-TL} holds  and there exists a positive real number $\mu$ such that
\begin{equation}
\label{LP5}
\left|\frac{\partial F}{\partial e}(0,x)\right|\leq \mu
\qquad \forall x\in \RR^{n_x}\ ,
\end{equation}
then Property \textsf{ULMTE} holds.
\end{proposition}
\vspace{0.2cm}

The proof of this proposition is given in Appendix \ref{Sec_ProofPropTensor}.
 The idea is to show that, for every symmetric positive definite
matrix $Q$, the function $\PR :\RR^{n_x}\rightarrow\RR^{n_e\times n_e}$ given by
\begin{equation}\label{eq_P}
\PR (\dx) = \lim_{T\to + \infty }\int_0^{T} \left(
\frac{\partial \dE}{\partial \de}(0,\dx,s)
\right)^\prime Q
\frac{\partial \dE}{\partial \de}(0,\dx,s)ds
\end{equation}
is well defined, continuous and satisfies all the requirements of the property ULMTE. 
The assumption (\ref{LP5}) 
is used to show that $P$ satisfies the left inequality
in (\ref{LP7}).
Nevertheless this inequality holds without (\ref{LP5}) provided the 
function
$s\mapsto \left|\frac{\partial \dE}{\partial \de}(0,\dx,s)\right|$ 
does not go too fast to zero.

\subsubsection{%
\startmodif
\textsf{ULMTE} ``$\Rightarrow$'' \textsf{TULES-NL}
\stopmodif
}
\begin{proposition}\label{Prop_Lyap}
If Property \textsf{ULMTE} holds and there exist positive
real numbers $\eta  $ and $c$ such that,
for all $(e,x)$ in $B_e(\eta )\times\RR^{n_x}$,
\begin{eqnarray}
\label{LP8}
&\displaystyle \left| \frac{\partial P}{\partial x} (x)\right|\leq c
\; ,
\\[0.3em]
\label{LP6}
&\hskip -1.6 em
\displaystyle \left|\frac{\partial^2 F}{\partial e\partial
e}(e,x)\right|\leq c\; ,\
\left|\frac{\partial^2 F}{\partial x\partial e}(e,x)\right|\leq c\;  ,\
\left|\frac{\partial G}{\partial e}(e,x)\right|\leq c
\, ,\null
\end{eqnarray}
then Property \textsf{TULES-NL} holds.
\end{proposition}
\vspace{0.2cm}

The proof of this proposition can be found
 in Appendix \ref{Sec_ProofPropLyap}. 
This is a direct consequence of the use of $V(e,x) = e^\prime P(x)e$ as a Lyapunov function.
The bounds (\ref{LP8}) and (\ref{LP6}) are used
 to show that, with equation (\ref{eq_TensorDerivative}),  the time derivative of this Lyapunov function is negative in a (uniform) tubular neighborhood of the manifold $\{(e,x),e=0\}$. 

\startmodifbj 
\subsection{Revisiting the exponential incremental stable systems}
\label{Sec_IncStability}
Incremental stability of an autonomous system 
(\ref{Aut_Syst}) is the property that a distance between any two solutions of  (\ref{Aut_Syst}) converges asymptotically to zero.  
The characterization of it has been studied thoroughly, for example, in \stopmodif  \cite{Angeli_TAC_02_lyapIncStab,FromionScorletti_CDC_05,Angeli_TAC_09_FurthReslyapIncStab}.
In \cite{Angeli_TAC_02_lyapIncStab, Angeli_TAC_09_FurthReslyapIncStab}, a Lyapunov characterization of
incremental stability ($\delta$-GAS for autonomous systems and
$\delta$-ISS for non-autonomous ones) is given based on the Euclidean
distance between two states that evolve in an identical system.
A variant of this notion is that of convergent systems 
discussed in
\cite{PavlovVanWouhNijmeijer_SCL_04ConvSystTribDemidovich,RufferEtAl_SCL_13ConvSystIncStab}.
All these studies are based on the notion of contracting flows which has been widely studied in the
literature and for a long time, 
see, for example,
\cite{Lewis_49_AJM_MetPropDiffEq,Lewis_51_AJM_DiffEqVarMet,Hartman_Book_64,
Demidovich_UspeMatNauk_61_DissSysDiffEq,Nemeth_PhD_98_GeomMinBroMonot,LohmillerSlotine_Aut_98_ContAnNLSyst,ForniSepculchre_TAC_2014}.
These flows generate trajectories between which an appropriately defined
distance  is monotonically decreasing with increasing 
time.   See
\cite{Jouffroy_CDC_05AncContAnal} for a historical discussion on the
contraction analysis and \cite{Sontag_Book_10ContrSystInput} for a partial survey.

The big issue in this view points is to find the appropriate distance 
which may be a difficult task.
The results in
Section \ref{sec_NormExpManifold} may help in this regard
with providing an explicit construction of
a  Riemannian distance.

Precisely, let $P$ be a $C^2$ 
function defined on $\RR^n$ the values of which are symmetric 
matrices satisfying
\begin{equation}
\label{LP10}
\underline{p}\,   I \leq  P ( x)\leq \overline{p}\,  I
\qquad \forall x \in \RR^n
\end{equation}
The length of  any piece-wise $C^1$ path $\gamma :[s_1,s_2]\to
\RR^n$ between two arbitrary points $x_1=\gamma (s_1)$ and 
$x_2=\gamma (s_2)$ in
$\RR^n$ is defined as~:
\begin{equation}\label{eq_RiemanianLength}
\left. L(\gamma)\vrule height 0.51em depth 0.51em width 0em \right|_{s_1}^{s_2}=\int_{s_1}^{s_2}\sqrt{\frac{d\gamma}{ds}(\sigma )^\prime  P (\gamma(\sigma ))\frac{d\gamma}{ds}(\sigma )}\: d\sigma
\end{equation}
By minimizing along all such path we get the distance  $d(x_1,x_2)$.

Then, thanks to
the well established
relation between (geodesically) monotone vector field (semi-group
generator) (operator) and contracting (non-expansive) flow (semi-group)
(see \cite{Lewis_49_AJM_MetPropDiffEq,Hartman_Book_64,Brezis_Book_73,IsacNemeth_Book_08} and many
others), we know that this distance between any two solutions of 
(\ref{Aut_Syst}) is exponentially decreasing to $0$ as time goes on 
forward if we have
\begin{equation}
\label{eq_TensorDerivIncremental}
\der_f P (x)+ P (x)\frac{\partial f}{\partial x}(x)
+\frac{\partial f}{\partial x}(x)^\prime  P (x) \leq  -Q
\qquad \forall x\in \RR^n
\  ,
\end{equation}
where $Q$ is a positive definite symmetric matrix and
\begin{equation}
\label{LP11}
\der_{f} P(x)\;=\;
 \lim_{h\to 0}
\frac{ P( X( x,h))- P( x)}{h}
\  .
\end{equation}
For a proof, see for example \cite[Theorem 1]{Lewis_49_AJM_MetPropDiffEq} or \cite[Theorems
5.7 and 5.33]{IsacNemeth_Book_08}
or \cite[Lemma 3.3]{Reich_Book_05_NLSemGrp} (replacing $f(x)$ by $x+ h
f(x)$).

In this context, using the main results
of our previous section, we can
 show that, if we have exponential incremental stability, then there exists
a function $P$ meeting the above requirements.
Specifically, we have the following proposition.\vspace{0.2cm}

\begin{proposition}[Incremental stability]
\label{theo_IncStab}
Assume the
system (\ref{Aut_Syst}) is forward complete with a function $f$ which 
is $C^3$ with bounded
first, second and third derivatives. Let $X(x,t)$ denotes its
solutions. Then we have P1 $\Rightarrow$ P2 
$\Rightarrow$ P3 (and therefore P1 $\Leftrightarrow$ P2 
$\Leftrightarrow$ P3).
\end{proposition}
\vspace{0.2cm}

In other words, exponential
incremental stability property is equivalent to the existence of a
Riemannian distance
which is contracted by the flow
 and can be used as a
$\delta$-GAS Lyapunov function.
\startmodifbj Note also that, despite the fact that the main results in Section \ref{sec_NormExpManifold} are local, when we restrict ourselves to the incremental stability problem, we can obtain a global result.\\[0.2cm] \stopmodifbj 

\begin{proof}
\underline{P1 $\Rightarrow$ P2 $\Rightarrow$ P3:}
\startmodif
Consider the system (\ref{LP36}) and let $n_x=n_e=n$.
\stopmodif
The boundedness of the first derivative of $f$ implies
the forward completeness of the corresponding systems
(\ref{eq_System}) and (\ref{LP4}).
Moreover the inequalities (\ref{LP1}), (\ref{LP2})  and (\ref{LP5})
with $r=+\infty $ follow from the assumption of boundedness of the 
derivatives of $f$. 

As a consequence
P1$\Rightarrow$ P2  follows from
Proposition \ref{Prop_DetecNec}
and P2$\Rightarrow$ P3  from
Proposition \ref{Prop_ExistTensor}.
Note however that it remains to show that $P$ defined in (\ref{eq_P}) is $C^2$.
This is obtained employing the boundedness of the first, second and third derivatives of $f$.
Indeed, note that we have for all $(t,x)$
 $
 \frac{\partial \dE}{\partial e}(0,x,t) = \frac{\partial X}{\partial x}(x,t)
 $.
 So to show that $P$ is $C^1$ it suffices to show that the mapping $t\mapsto \frac{\partial^2 X}{\partial x_i\partial x}(x,t)$ goes exponentially to zero as time goes to infinity.
Note that this is indeed the case since given a vector $v$ in $\RR^n$ and $i$ in $\{1,\dots,n\}$ the mapping $\nu(t)= \frac{\partial^2 X}{\partial x_i\partial x}(x,t)v$ is solution to
\begin{multline*}
\dot \nu = \frac{\partial f}{\partial x}(X(x,t))\nu+\\\sum_{j=1}^n\frac{\partial^2 f}{\partial x_j\partial x}(X(x,t))\frac{\partial X_j}{\partial x_i}(x,t)\frac{\partial X}{\partial x}(x,t)v
\end{multline*}
Hence, from (\ref{eq_TensorDerivIncremental}), (\ref{LP10}) and the fact that $f$ has bounded second derivatives, it yields the existence of a positive real number $\tilde c$ such that
$$
\dot{\overparen{\nu^\prime P(X(x,t))\nu}}\leq -\nu^\prime Q\nu + \tilde c|\nu| \left|\frac{\partial X}{\partial x}(x,t)\right|^2\ .
$$
Since $t\mapsto \frac{\partial X}{\partial x}(x,t)$ exponentially goes to zero as time goes to infinity, it implies that $\nu$ exponentially goes to zero. Hence, $P$ is $C^1$. Employing the bound on the third derivative and following the same route, it follows that $P$ is $C^2$.
\end{proof}

\section{Applications}
\label{Sec_Applications}
In this section, we apply Propositions \ref{Prop_DetecNec}, \ref{Prop_ExistTensor} and \ref{Prop_Lyap} in
\startmodif
two
\stopmodif
different contexts: full order observer and synchronization.

\subsection{Nonlinear observer design}
\label{Sec_Observer}
\startmodif
Consider a system 
\begin{equation}\label{Aut_Syst_Output}
\dot x = f(x)\ ,\quad y=h(x)\ .
\end{equation}
with state $x$ in $\RR^n$ and output $y$ in $\RR^p$
augmented with  a state observer of the particular form
\begin{equation}\label{Observer_main_eq}
\dot {\hat x} = f(\hat x) + K(y,\hat x) 
\end{equation}
with state $\hat x$ in $\RR^n$ and where
\begin{equation}
\label{LP15}
K(h(x),x)\;=\; 0\qquad \forall x
\  .
\end{equation}
Assuming the functions $f$, $h$ and $K$ are $C^2$, we are interested in having
the manifold $\{(x,\hat x):\,  x=\hat x\}$ exponentially stable for the overall system 
\begin{equation}\label{eq_Observer}
\dot x = f(x) \ ,\quad\dot {\hat x} = f(\hat x) + K(y,\hat x)\ .
\end{equation}

When specified to this context, the properties
TULES-NL,  UES-TL and ULMTE are
\begin{list}{}{%
\parskip 0pt plus 0pt minus 0pt%
\topsep 1ex plus 0pt minus 0pt%
\parsep 0.5ex plus 0pt minus 0pt%
\partopsep 0pt plus 0pt minus 0pt%
\itemsep 1ex plus 0pt minus 0pt
\settowidth{\labelwidth}{1em}%
\setlength{\labelsep}{0.5em}%
\setlength{\leftmargin}{\labelwidth}%
\addtolength{\leftmargin}{\labelsep}%
}
\item[\textsf{Exponentially convergent observer (TULES-NL)}:]
The system (\ref{eq_Observer}) is forward complete and there exist strictly positive real numbers $r$, $k$ and
$\lambda$ such
that we have, for all $(x,\hat x,t)$ in $\RR^n\times\RR^{n_x}\times
\RR_{\geq 0}$ satisfying
$ |x_0-\hat x_0|\leq r$, we have
\begin{equation}\label{eq_Observer_convrate}
|X(x_0,t) - \hat X(x_0,\hat x_0,t)|\leq k|x_0-\hat x_0|\exp(-\lambda t) 
\  .
\end{equation}
\item[\textsf{UES-TL  FOR OBSERVER}]
The system
$$
\dot x= f(x)
$$
is forward complete and there exist strictly positive real numbers $\dk $  and $\tilde \lambda$
such that any solution $(\dE(\de_0,x_0,t),X(x_0,t))$  of the transversally linear system
\begin{equation}\label{eq_LinObsSysK}
\dot \de = \left[\frac{\partial f}{\partial x}(x) + \frac{\partial K}{\partial y}(h(x),x)\frac{\partial h}{\partial x}(x)\right] \de
\ ,\quad\dot x = f(x)
\end{equation}
satisfies, for all $(\de_0,x_0,t)$ in $\RR^{n}\times\RR^{n}\times
\RR_{\geq 0}$,
\begin{equation}
\label{LP12}
|\dE(\de_0,x_0,t)|\leq  \dk \exp(-\tilde \lambda t)|\de_0|
\ .
\end{equation}

%
\item[\textsf{ULMTE FOR OBSERVER}]:
For all positive definite matrix $Q$, there exists a continuous
function $\PR :\RR^{n}\rightarrow\RR^{n\times n}$
and strictly positive real numbers $\underline{p}$ and $\overline{p}$
such that we have, for all $x$ in $\RR^{n}$,
$$
\underline{p} \,  I \leq  \PR (x)\leq \overline{p}\,  I
\  ,
$$
\vbox{\hsize=\linewidth\noindent
$\displaystyle 
\der_f \PR(x) 
$\hfill \null \\\null \hfill   $\displaystyle 
+\; 
2 \texttt{Sym}\left( \!\PR (x ) \!\left[\frac{\partial f}{\partial x}(x) + \frac{\partial K}{\partial y}(h(x),x)\frac{\partial h}{\partial x}(x)\right]
\!\right)
$\refstepcounter{equation}\label{LP13}\hfill$(\theequation)$
\\[0.5em]\null \hfill $\displaystyle 
\leq -Q
\  .
\quad $}\\[0.5em]
where $\texttt{Sym}(A)=A + A^\prime$.
\end{list}

Propositions \ref{Prop_DetecNec}, \ref{Prop_ExistTensor} and 
\ref{Prop_Lyap} give conditions under which these properties are 
equivalent. But these properties assume the data of the 
correction term $K$. Hence, by rewriting UES-TEL and TULES-NL
in a way in which the design parameter $K$ disappears,
these propositions give necessary conditions for the existence of an exponentially 
convergent observer.

Property UES-TL involves the  existence of an observer with correction term depending on $x$
for the time-varying linear system
resulting from the linearization along a
solution
 to the system (\ref{Aut_Syst_Output}), i.e.
\begin{equation}\label{eq_LinObsSys}
\dot \de = \frac{\partial f}{\partial x}(x)\de \ ,\quad\tilde y = \frac{\partial h}{\partial x}(x)\de
\end{equation}
seeing $\tilde y$ as output.
As a consequence of Proposition \ref{Prop_DetecNec}, a necessary condition for Property UES-TL to hold 
and further, when some derivatives are bounded, for the existence of an exponentially convergent observer is that the system
(\ref{Aut_Syst_Output}) be infinitesimally detectable in the 
following sense
\begin{list}{}{%
\parskip 0pt plus 0pt minus 0pt%
\topsep 1ex plus 0pt minus 0pt%
\parsep 0.5ex plus 0pt minus 0pt%
\partopsep 0pt plus 0pt minus 0pt%
\itemsep 1ex plus 0pt minus 0pt
\settowidth{\labelwidth}{1em}%
\setlength{\labelsep}{0.5em}%
\setlength{\leftmargin}{\labelwidth}%
\addtolength{\leftmargin}{\labelsep}%
}
\item[\textbf{Infinitesimal detectability}]
We say that the system (\ref{Aut_Syst_Output}) is infinitesimally detectable 
if every solution of 
$$
\dot x =f(x)\ ,\quad\dot \de = \frac{\partial f}{\partial x}(x)\de \ ,\ \frac{\partial h}{\partial x}(x)\de=0
$$ 
defined on $[0,+\infty)$ satisfies $\lim_{t\rightarrow+\infty}|\dE(e,x,t)|=0$.
\end{list}
A similar necessary condition has been established in \cite{AndrieuBesanconSerres_CDC_13_NecObs}
for a larger class of observers but under an extra assumption (the existence of a locally quadratic Lyapunov function).

\par\vspace{1em}
\noindent \textbf{Example 1:}
{\itshape
Consider the planar system
\begin{equation}\label{eq_ExpObsSyst}
\dot x_1 \;=\;  x_2^3\ ,\quad\dot x_2 \;=\;  -x_1\ ,\quad y\;=\; x_1\ .
\end{equation}
We wish to know whether or not it is possible to design an 
exponentially convergent observer for this nonlinear oscillator in 
the form of (\ref{Observer_main_eq}).
The linearized system is
\begin{equation}
\label{LP14}
\dot \de_1 \;=\;  3x_2^2 \,  \de_2\ ,\quad\dot \de_2 \;=\;  -\de_1\ ,\quad  \tilde y \;=\;  \de_1
\end{equation}
This system is not detectable when the solution, along which we 
linearize, is the origin which is an equilibrium of (\ref{eq_ExpObsSyst}).
Consequently the system (\ref{eq_ExpObsSyst}) is not infinitesimally 
detectable on $\RR^2$ and so there is no exponentially convergent 
observer on $\RR^2$.
Fortunately the subset $\{x\in\RR^2:\,  \frac{x_1^2}{2} + 
\frac{x_2^4}{4} \geq \epsilon\}$, with $\epsilon>0$, is invariant and 
(\ref{eq_ExpObsSyst}) is infinitesimally detectable in it.

To design a correction term $K$ for an exponentially convergent 
observer, we use the property that
$$
L(y,x)\;=\; \frac{\partial K}{\partial y}(y,x)
$$
should be an observer gain for the linear system (\ref{LP14}). So we 
start our design by selecting $L$. We pick
$$
L(y,x) = \begin{bmatrix}
-3x_2^2\\
-3x_2^2 +1
\end{bmatrix} 
$$
This gives (see (\ref{eq_LinObsSysK}))
$$
A(x)\;=\; \frac{\partial f}{\partial x}(x) + \frac{\partial K}{\partial y}(h(x),x)\frac{\partial h}{\partial x}(x)
\;=\; 3 x_2 \begin{bmatrix}
-1 & 1\\
0 & -1 
\end{bmatrix} 
$$
The transition matrix generated by $A(X(x,t))$ when $X(x,t)$ is a 
solution of (\ref{eq_ExpObsSyst}) is
$$
\Phi(t,0)\;=\; \exp\left(-I(t)\right)
\begin{bmatrix}
1 & I(t)\\
0 & 1
\end{bmatrix} 
$$
where
$$
I(t)\;=\; 3\,  \int_0^t X_2(x,s)^2 ds
\  .
$$
Since $X_2(x,t)$ is periodic, (\ref{LP12}) holds when the initial 
condition $x$ is in the compact invariant subset
\begin{equation}
\label{LP16}
\mathcal{C}\;=\; \{x\in\RR^2:\,  \frac{1}{\varepsilon }\geq \frac{x_1^2}{2} + 
\frac{x_2^4}{4} \geq \epsilon\}
\  .
\end{equation}
Then, according to
Propositions \ref{Prop_ExistTensor} and 
\ref{Prop_Lyap}, and in view of (\ref{LP15}), we obtain an exponentially 
convergent observer by choosing $K$ as
$$
K(y,x)\;=\; \int_{x_1}^y L(s,x) ds\;=\; \begin{bmatrix}
-3x_2^2\\
-3x_2^2 +1
\end{bmatrix} (y-x_1)
\  .
$$
}
\par\vspace{1em}
Similarly, Property ULMTE involves the  existence of $P$ and $K$ such that 
(\ref{LP13}) holds. By restricting this inequality on quadratic 
forms to vectors which are in the kernel of $\frac{\partial 
h}{\partial x}$, we obtain as a consequence of Propositions 
\ref{Prop_DetecNec} and \ref{Prop_ExistTensor}  that a necessary 
condition for Property ULMTE  to hold 
and further, when some derivatives are bounded, for the existence of an exponentially convergent observer is that the system
(\ref{Aut_Syst_Output}) be R-detectable (R for Riemann) in the 
following sense.
\begin{list}{}{%
\parskip 0pt plus 0pt minus 0pt%
\topsep 1ex plus 0pt minus 0pt%
\parsep 0.5ex plus 0pt minus 0pt%
\partopsep 0pt plus 0pt minus 0pt%
\itemsep 1ex plus 0pt minus 0pt
\settowidth{\labelwidth}{1em}%
\setlength{\labelsep}{0.5em}%
\setlength{\leftmargin}{\labelwidth}%
\addtolength{\leftmargin}{\labelsep}%
}
\item[\textbf{R-Detectability}]
We say that the system (\ref{Aut_Syst_Output}) is R-detectable if
there
 exist  a continuous function $\PR :\RR^n\rightarrow\RR^{n\times n}$ and  positive
real numbers
 $0<\underline{p}\leq \overline{p}$
and $0<\underline{q}$ such that 
$P$ has a derivative $\der_{f}P$ along $f$ in the
sense of (\ref{LP11}) and we have
\begin{equation}
\label{LP17}
\underline{p}\,   I \leq  P ( x)\leq \overline{p}\,  I
\qquad \forall x \in \RR^n
\end{equation}
and
\begin{equation}\label{eq_ContrTangentOutputFunction}
v^\prime \der_f\PR (\dx)v +
2v^\prime \PR (\dx)\frac{\partial f}{\partial x}(x)v \leq  
-\underline{q}\,   v^\prime P(x) v
\end{equation}
 holds for all $(x,v)$ in $\RR^n\times \RR^n$ satisfying
$\frac{\partial h}{\partial x}(x)v=0$. 
\end{list}
A similar necessary condition has been established in 
\cite{SanfelicePraly_TAC_12}, where only 
asymptotic and not exponential convergence is assumed. In that case, 
the condition allows $\underline{p}$ and $q$ to be zero.


Further it is established in \cite{SanfelicePraly_Report_13RiemanObs} that when the R-detectability holds 
then
$$
K(y,x)\;=\; k\,  P(x)^{-1}\frac{\partial h}{\partial x}(x)^T (y-h(x))
$$
gives, for $k$ large enough, a (locally) exponentially convergent observer.
\par\vspace{1em}
\noindent \textbf{Example 1 continued:}
{\itshape
For the system (\ref{eq_ExpObsSyst}), the necessary R-Detectability 
condition is the existence of
$P=\left(\begin{array}{cc}
P_{11} & P_{12} \\ P_{12} &P_{22}
\end{array}\right)$
satisfying in particular  (\ref{eq_ContrTangentOutputFunction}) which 
is (see (\ref{LP16}))
\begin{equation}
\label{LP19}
\frac{\partial P_{22}}{\partial x_1} (x)\,  x_2^3
-\frac{\partial P_{22}}{\partial x_2} (x)\,  x_1
+ 6 P_{12} (x)\,  x_2^2 < -\underline{q}\,  P_{22}(x)
\quad \forall x\in \mathcal{C}
\end{equation}
We view this as a condition on $P_{22}$ only since whatever $P_{12}$ 
is, we can always pick $P_{11}$ to satisfy (\ref{LP17}). Note also 
that we can take care of any term with $x_2^2$ in factor by selecting 
$P_{12}$ appropriately. With this, it can 
be shown that it is sufficient to pick $P_{22}$ in the form
$$
\left( \  \frac{r(x)}{2}\; \leq \right)\quad P_{22}(x)\;=\; r(x) + \frac{x_1x_2}{\sqrt{r(x)}} + x_2^2
\quad \left(\   \leq \; 3 r(x)\right)
\  ,
$$
where the presence of $r$ defined below is justified by homogeneity 
considerations
$$
\left( \  4\varepsilon \; \leq \right)\quad
r(x)\;=\; \sqrt{2x_1^2 +x_2^4}
\quad \left(\   \leq \; \frac{4}{\varepsilon }\right)
\  .
$$
This motivates  \startmodifbj us to design \stopmodifbj 
$$
P_{12}(x)\;=\; -\frac{5}{24}\frac{x_2^2}{\sqrt{r(x)}} - \frac{1}{3} \sqrt{r(x)}
$$
In this case, the left hand side in the inequality (\ref{LP19}) is
\\[0.5em]$\displaystyle 
-\frac{x_1^2}{\sqrt{r}}
-\frac{1}{4}\frac{x_2^4}{\sqrt{r(x)}}
+2x_1x_2
-2 \sqrt{r(x)}x_2^2
$\hfill \null \\\null \hfill $\displaystyle 
\leq\, -\frac{1}{4} r(x) \sqrt{r(x)}
\, \leq\, -\frac{\sqrt{r(x)}}{12} P_{22}(x)
\, \leq \, -\frac{\sqrt{\varepsilon }}{6} P_{22}(x)
\  .
$
\par\vspace{1em}
Finally, by choosing
$$
P_{11}(x)\;=\; 2 +\frac{P_{12}(x)^2}{P_{22}(x)-\frac{r(x)}{4}}
\  ,
$$
it can be shown that we obtain
$$
\varepsilon \,  I\, \leq \, \min\left\{\!1, \frac{r(x)}{4}\!\right\}  I\, \leq \, P(x)\, \leq \,  
3\,  \max\{1,r(x)\}\,  I\, \leq \, \frac{12}{\varepsilon }
\:  .
$$
Hence (\ref{LP17}) holds on $\mathcal{C}$.
From this, the correction term
$$
K(y,x)\;=\; \frac{kP_{22}(x)}{P_{11}P_{22}(x)-P_{12}(x)^2}\,  \left(\begin{array}{c}
P_{22}(x) \\ -P_{12}(x)
\end{array}\right) (y-x_1)
$$
gives a (locally) convergent observer on $\mathcal{C}$.
}

\stopmodif
\subsection{Exponential synchronization}
\label{Sec_Synchronization}

 Finally, we revisit the synchronization problem as another class of control problems that can be dealt
with the
results in Section II. 
We consider here the synchronization of $m\geq 2$ identical systems given by  
\begin{equation}\label{eq_Syst}
\dot w_i=f(w_i)+g(w_i)u_i\ , \ i=1,\dots, m\ ,
\end{equation}
 In this setting, all systems  have the same drift vector field $f$ and the same control vector field $g:\RR^n\rightarrow\RR^{n\times p}$, but not the same controls in $\RR^p$. 
The state of the whole system is denoted $w=(w_1,\dots w_m)$ in $\RR^{mn}$. We define also the diagonal subset of $\RR^{mn}$
$$
\DR = \{(w_1,\dots, w_m)\in\RR^{mn}, w_1=w_2 \dots =w_m\}
$$
Given $w$ in $\RR^{mn}$, we denote the Euclidean distance to the set $\DR$ as $|w|_\DR$.
The  synchronization problem  that we consider in this section  is as follows. 
\begin{definition}\label{Def_Synchronization}
The control laws 
$u_i=\phi_i(w)$, $w=(w_1,\dots w_m)$, $i=1\dots, m$ solve the \startmodif \emph{local uniform exponential synchronization problem} for (\ref{eq_Syst}) if the following holds: \stopmodif
\begin{enumerate}
\item  $\phi$ is invariant by permutation of agents. More precisely, given a permutation $\pi:\{1,\dots m\}\mapsto \{1,\dots m\}$
$$
\phi_{\pi_i}(w_1,\dots,w_m) =  \phi_i(w_{\pi_1}, \dots, w_{\pi_m})
$$
\item $\phi$ is zero on $\DR$:
\begin{equation}
\label{LP22}
\phi(w)=0 \qquad \forall w \in \DR
\; ,
\end{equation}
\item and the 
set
 $\DR$ 
is uniformly exponentially stable
for the closed-loop system,
i.e., there exist positive 
real numbers
 $r_w$, $k$ and $\lambda>0$ such that, for all $w$ in $\RR^{mn}$ satisfying $|w|_\DR<r_w$,  
\begin{equation}\label{eq_Beta_exp}
|W(w,t)|_\DR  
\leq  k\exp(-\lambda t)\,  |w|_\DR,
\end{equation}
holds for all $t$ in the domain of existence of the solutions $W(w,t)$ 
going through
$w$ at $t=0$. 
\end{enumerate}When $r_w=\infty$, it is called the global uniform exponential synchronization problem.   
\end{definition}

In this context, we assume that 
 every agent shares an information \startmodif (which will be designed later) \stopmodif to all other agents (in which case, it forms a complete graph) and it has local access to its state variables. 

 \startmodif
It is possible to rewrite the property  of having the manifold $\DR$ exponentially stable as property TULES-NL.
As it has been done in the observer design context, employing Propositions \ref{Prop_DetecNec} and \ref{Prop_ExistTensor} and by rewriting properties UES-TL and ULMTE it is possible to give equivalent characterization of the synchronization property. 
By rewriting these conditions in a way in which the control law disappears, these properties give necessary
 conditions to achieve exponential synchronization.
 \stopmodif
\begin{proposition}[Necessary condition]
\label{prop2}
Consider  the systems in (\ref{eq_Syst}) and assume  the existence of control laws $u_i=\phi_i(w)$, $i=1,\dots m$ that solve the uniform exponential synchronization of (\ref{eq_Syst}).  
 Assume moreover that
$g$ is bounded and
 $f$, $g$ and the $\phi_i$'s have bounded first and second derivatives.
Then the following two properties hold. 
\begin{enumerate}
\item[Q1:] The origin of the transversally linear system
\begin{equation}
\dot \de = \frac{\partial f}{\partial x}(\dx)\de + g(\dx)u\ ,\ \dot
\dx = f(\dx)\ ,
\end{equation}
is stabilizable by a (linear in $\de$) state feedback.
\item[Q2:]
For every positive definite matrix $Q$, there exist a continuous
function $\PR :\RR^{n}\rightarrow\RR^{n\times n}$
and positive real numbers $\underline{p}$ and $\overline{p}$
such that inequalities (\ref{LP10})
are satisfied,
$P $ has a derivative $\der_{f}P$ along $f$ in the
sense of (\ref{LP11}),
and 
\begin{equation}\label{eq_TensorDerivSynch}
\der_fv^\prime \PR (x)v + 2v^\prime \PR (x)\frac{\partial f}{\partial x}(x)v \leq  -v^\prime Qv\ 
\end{equation}
holds for all $(v,x)$ in $\RR^{2n}$ satisfying $v^\prime \PR(x)g(x)=0$. 
\end{enumerate}
\end{proposition}

\begin{proof}
First of all note that the vector fields having bounded first derivatives, it implies that the system is complete.
Consider $(i,j,k,l)$, $4$ integers  in $\{1,\dots, m\}$ and consider a permutation $\pi:\{1,\dots,m\}\mapsto \{1,\dots,m\}$ 
such that $\pi_i=k$ and $\pi_j=\ell$.
Note that $k=\ell$ if and only if $i=j$.
Note that the invariance by permutation implies
$$
\phi_k(w) = \phi_i(w_{\pi_1}, \dots, w_{\pi_m})\ .
$$
Hence, it follows that
$$
\frac{\partial \phi_k}{\partial w_\ell}(w) = \frac{\partial \phi_i}{\partial w_j}(w_{\pi_1}, \dots, w_{\pi_m}) 
$$
and if we consider $w$ in $\DR$, this implies
\begin{align*}
&\frac{\partial \phi_k}{\partial w_\ell}(w) = \frac{\partial \phi_i}{\partial w_j}(w) \ , i\neq j\ , \ k\neq \ell\\
&\frac{\partial \phi_j}{\partial w_j}(w) = \frac{\partial \phi_i}{\partial w_i}(w)
\end{align*}
 By denoting $e= (e_2, \dots e_m)$ with  $e_i=w_i-w_1$, $i=2, \dots m$ and $x=w_1$, 
we obtain an $(e,x)$-system of the type (\ref{eq_System}) with
\begin{eqnarray}
\label{eq_VecFieldSynchro1}F(e,x)&\hskip -0.8em =&(F_i(e,x))_{i=2,\dots m}\\
F_i(e,x)&\hskip -0.8em =&\hskip -0.8em
\nonumber f(x+e_i)-f(x)
\\
\label{eq_VecFieldSynchro2}&\hskip -0.8em
&\hskip -0.8em
+ g(x+e_i)\bar \phi_i(e,x)-g(x)\bar \phi_1(e,x)\ ,
\\[0.3em]
\label{eq_VecFieldSynchro3}G(e,x)&\hskip -0.8em
=&\hskip -0.8em
 f(x)+g(x)\bar \phi_1(e,x) \ ,
\end{eqnarray}
where we have used the notation
$$
\bar \phi_i(e,x) = \phi_i(x,x+e_2,\dots , x+e_n)
$$
Note that we have
\begin{align}
 |e|^2 &
\label{eq_EqNorme1}
\leq (m-1)|w|_\DR^2\ ,
\end{align}
and
\begin{align}
\nonumber |w|_\DR^2 
&\leq |e|^2 +(m-1)\left |\sum_{i=1}^m \frac{w_1- w_i}{m}\right |^2 \\
\label{eq_EqNorme2} &\leq\left  (1+\frac{m-1}{m^2}\right )|e|^2
\end{align}
Hence, (\ref{eq_Beta_exp}) implies for all $(e,x)$ with $|e|\leq \frac{mr_w}{\sqrt{m^2 + m-1}}$
$$
|E(e,x,t)|\leq  \sqrt{(m-1)\left  (1+\frac{m-1}{m^2}\right )} k\exp(-\lambda t)\,  |e|,
$$
It follows from the assumptions  of the proposition  that Property \textsf{TULES-NL} is
satisfied with $r=\frac{mr_w}{\sqrt{m^2 + m-1}}$ and that inequalities (\ref{LP1}) and
(\ref{LP2}) hold. We conclude with Proposition \ref{Prop_DetecNec} that Property  \textsf{UES-TL}  is
satisfied also. 
So, in particular,
there exist positive real numbers $\dk$ and
$\tilde \lambda$
such that any $e_i$ component of $(\dE(\de_0,\dx_0,t),\dX(\dx_0,t))$  
solution of (\ref{eq_System_dif})
satisfies, for all $(\de_0,\dx_0,t)$ in $\RR^{mn}\times\RR^{mn}\times
\RR_{\geq 0}$,
\begin{equation}\label{eq_dei}
|\dE_i(\de_0,\dx,t)|\leq  \dk \exp(-\tilde \lambda t)|\de_0|
\ .
\end{equation}
On another hand, with (\ref{LP22}), we obtain~:
\begin{equation}\label{eq_dFidi}
\frac{\partial F_i}{\partial e_i}(0,\dx)\;=\;
\frac{\partial f}{\partial x}(\dx)
+
g(\dx) \left[\frac{\partial \bar\phi_i}{\partial e_i}(0,\dx)
-
\frac{\partial \bar\phi_1}{\partial e_i}(0,\dx)
\right]\ .
\end{equation}
And, when $j\neq i$, it yields
\begin{equation}\label{eq_dFidj}
\frac{\partial F_i}{\partial e_j}(0,\dx)\;=\;
g(\dx) \left[\frac{\partial \bar\phi_i}{\partial e_j}(0,\dx)
-
\frac{\partial \bar\phi_1}{\partial e_j}(0,\dx)
\right]=0\ .
\end{equation}
Consequently,
any solution of the system
$$
\dot \de_i \;=\; \left [\frac{\partial f}{\partial x}(\dx)
+
g(\dx) \left[\frac{\partial \bar\phi_i}{\partial e_i}(0,\dx)
-
\frac{\partial \bar\phi_1}{\partial e_i}(0,\dx)
\right]\right ]\de_i\ ,
$$
and
$
\dot \dx = f(\dx)
$
can be expressed as an
$e_i$ component of $(\dE(\de_0,\dx_0,t),\dX(\dx_0,t))$  
solution of (\ref{eq_System_dif})
Since these solutions satisfy (\ref{eq_dei}),
Property Q1 does hold.

Finally
we consider the system with state $(e_i,x)$ in $\RR^{2n}$
\begin{equation}\label{eq_Sysei}
\dot e_i = \bar F_i(e_i,x)\ ,\ \dot x = G(e_i,x) = f(x)
\end{equation}
with $\bar F_i(e_i,x) = F_i((0,e_i,0),x)$.
The previous property and Proposition \ref{Prop_ExistTensor} imply
that Property \textsf{ULMTE} is satisfied for system (\ref{eq_Sysei}). So in particular we
have a function $P$ satisfying the properties in Q2 and
such that we have, for all $(v,x)$ in $\RR^n\times \RR^n$,
\\[0.7em]$\displaystyle
v^\prime \der_{f} \PR(\dx) v
$\hfill \\[0.3em]\null \hfill $\displaystyle
+
2 v^\prime \PR (\dx )
\left(\frac{\partial f}{\partial x}(\dx)
+
g(\dx) \left[\frac{\partial \bar\phi_i}{\partial e_i}(0,\dx)
-
\frac{\partial \bar\phi_1}{\partial e_i}(0,\dx)
\right]\right) v
$\hfill \\[0.3em]\null \hfill $\displaystyle
\leq -v^\prime Q v
$\\[0.7em]
which implies (\ref{eq_TensorDerivSynch})  when $v^\prime \PR (x)g(x)=0$.
\end{proof}

\startmodif
\noindent \textbf{Example 2:} \stopmodif
{\it
As an illustrative example consider the case in which the system is
given by by $m$ agents $w_i$ in $\RR^2$
with individual dynamics
\begin{equation}\label{eq_exemple}
\dot w_{i1} = w_{i2} + 2 \sin(w_{i2})\ ,\ \dot w_{i2} = a + u_i\ ,
\end{equation}
where $a$ is a real number.
Because of a singularity when $1+2\cos(w_{i2})=0$,
this system is not feedback linearizable per se.
Hence the design of a synchronizing controller may be involved.
}

{\it
In order to check if local synchronization in the sense of Definition \ref{Def_Synchronization} is possible, the  necessary conditions of Proposition \ref{prop2} may be tested.
The transversally linear system is
\begin{equation}\label{eq_exp}
\dot \de = \begin{bmatrix} 0 & 1+2\cos(x_0+a t)\\0&0 \end{bmatrix}\de + \begin{bmatrix}0\\1\end{bmatrix}u\ .
\end{equation}
When $a=0$ and $x_0=\frac{2\pi}{3}$, this system is not stabilizable 
by any feedback law. Hence in this case, with Proposition 
\ref{prop2}, there is no exponentially synchronizing control law in 
the sense of Definition \ref{Def_Synchronization}
satisfying (\ref{LP22}) in particular.
\par\vspace{1em}
}

Similar to the analysis of incremental stability in the previous
section
 and observer design in \cite{SanfelicePraly_TAC_12} ,
by using a function $\PR $ satisfying the property Q2 in Proposition \ref{prop2},  we can obtain sufficient conditions for the solvability of uniform exponential synchronization of (\ref{eq_Syst}). 

We do this under an extra assumption which is that, up to a scaling factor, the control vector field $g$ is a gradient field with $P$ as Riemannian metric.

\begin{proposition}[Local sufficient condition]
\label{prop3}
 Assume $f$  has bounded first and second derivatives,
and
$g$ is bounded and has bounded first and second derivatives. Moreover, assume that
\begin{list}{}{%
\parskip 0pt plus 0pt minus 0pt%
\topsep 0pt plus 0pt minus 0pt
\parsep 0pt plus 0pt minus 0pt%
\partopsep 0pt plus 0pt minus 0pt%
\itemsep 0pt plus 0pt minus 0pt
\settowidth{\labelwidth}{1.}%
\setlength{\labelsep}{0.5em}%
\setlength{\leftmargin}{\labelwidth}%
\addtolength{\leftmargin}{\labelsep}%
}
\item[1.]
there exist a $C^2$ function $U:\RR^n\to \RR$ 
and a bounded $C^2$ function $\alpha :\RR^n\to \RR^p$ which has bounded first and second derivative such
that
\begin{equation}
\label{LP24}
\frac{\partial U}{\partial x}(x)^\prime = P(x)g(x)\alpha (x)
\; ;
\end{equation}
holds for all $x$ in $\RR^n$; and
\item[2.]
there exist a positive definite matrix $Q$, a $C^2$ function $\PR :\RR^n\rightarrow\RR^{n\times n}$
with bounded derivative, and  positive
real numbers
$\underline{p}$, $\overline{p}$  and $\rho >0$
such that (\ref{LP10})
is satisfied
 and 
\\[0.7em]\null \hfill $\displaystyle 
v^\prime \der_f\PR (x)v +
2v^\prime \PR (x)\frac{\partial f}{\partial x}(x)v
- \rho \left|
\frac{\partial U}{\partial x}(x)
v\right|^2\leq  -v^\prime Qv
\  ,
$\\\null \hfill \refstepcounter{equation}\label{LP23}$(\theequation)$
\\[0.3em]
 holds for all $(x,v)$ in $\RR^n\times \RR^n$. 
\end{list}
Then there exist 
a real number
 $\underline{\kell}$  
 such that with the control laws $u_i=\phi_i(w)$ given by 
\\[0.7em]\null \hfill $\displaystyle 
\phi_i(w)=\kell\alpha (w_i)\,  \left [\sum_{j=1}^m \frac{U(w_j)}{m}-U(w_i)\right ]
$ \hfill \refstepcounter{equation}\label{eq_SyncConLaw}$(\theequation)$
\\[0.7em]
and $\kell\geq \underline{\kell}$ and if the closed loop system is complete then the local uniform exponential synchronization of (\ref{eq_Syst}) is solved. 
\end{proposition}
\par\vspace{1em}
Note that, for its implementation, the control law (\ref{eq_SyncConLaw}) 
requires that each agent $i$ communicates $U(w_i)$ to all the
other agents.
\begin{proof}
First of all, note that the control law $\phi_i$ is invariant by permutation due to its structure.
Let $e=(e_2,\dots, e_m)$ with $e_i=x_1-x_i$ and $x=x_1$.
We obtain an $(e,x)$-system of the type (\ref{eq_System}) 
%
with $F$ and $G$
as
defined in (\ref{eq_VecFieldSynchro1}-\ref{eq_VecFieldSynchro2}-\ref{eq_VecFieldSynchro3}) with $\phi$ as control input.
For this system, we will show  that property ULMTE is satisfied.
Consider the function $P_m:\RR^n\rightarrow\RR^{(m-1)n\times(m-1)n}$ defined as a block diagonal matrix composed of $(m-1)$ matrices $P$. 
i.e. $P_m(x) = \texttt{Diag}(P(x),\dots,P(x))$.
Note that with property (\ref{eq_dFidi}) and (\ref{eq_dFidj}), it yields that $
\frac{\partial F}{\partial e}(0,\dx)$ is also $(m-1)$ block diagonal.
Hence,
we have
\\[0.7em]$\displaystyle 
\der_{\dG} \PR_m(\dx) +
\PR_m (\dx ) \frac{\partial F}{\partial e}(0,\dx)
+ \frac{\partial F}{\partial e}(0,\dx)^\prime \PR_m(\dx) 
$\hfill \null \\\null \hfill $\displaystyle 
=
\texttt{Diag}\{R(\dx)),\dots, R(\dx))\}
$\\[0.7em]
where 
\\[0.7em]$\displaystyle 
R(\dx) =  \der_f P(\dx) v +  P(\dx)\left [\frac{\partial f}{\partial x}(\dx)
-\ell
g(\dx) \alpha(\dx)\frac{\partial U}{\partial x}(\dx)
\right ] 
$\hfill \null \\\null \hfill $\displaystyle +
\left [\frac{\partial f}{\partial x}(\dx)
-\ell
g(\dx) \alpha(\dx)\frac{\partial U}{\partial x}(\dx)
\right ] ^\prime P(\dx)
\  .
$\\[0.7em]
With (\ref{LP23}), this gives
\\[0.7em]\null \hfill $\displaystyle 
v^\prime
R(\dx)
 v \leq -v^\prime Qv + (k-2\ell) \left |\frac{\partial U}{\partial x}(\dx)
v
\right |^2\ .
$\hfill \null \\[0.7em]
for all $(\dx,v)$ in $\RR^n\times \RR^n$. 
Hence, picking $\ell>\frac{k}{2}$, inequality (\ref{eq_TensorDerivative}) holds.
To apply proposition \ref{Prop_Lyap}, it remains to show that inequalities (\ref{LP5}), (\ref{LP8}) and (\ref{LP6}) are satisfied.
Note that employing the bounds on the functions $P$, $f$, $g$, $\alpha$ and there derivatives, it is possible to get a positive real number $\tilde c$ (depending on $\ell$) such that for all $i$ in $2,\dots, m$ and all $(e,x)$
\begin{eqnarray*}
\left|\frac{\partial F_i}{\partial e_i}(e,x)\right| &\leq &\tilde c\left |\sum_{j=1}^m\frac{U(x+e_i)}{m}-U(x+e_i)\right | +\tilde c\ ,\ \\
\left|\frac{\partial F_i}{\partial e_i\partial x}(e,x)\right| &\leq &\tilde c\left |\sum_{j=1}^m\frac{U(x+e_i)}{m}-U(x+e_i)\right | +\tilde c\ ,\\
\end{eqnarray*}
\begin{eqnarray*}
\left|\frac{\partial F_i}{\partial e_i\partial e_j}(e,x)\right| &\leq&\tilde c\left |\sum_{j=1}^m\frac{U(x+e_i)}{m}-U(x+e_i)\right | +\tilde c\ ,\ 
\end{eqnarray*}
$$
\left |\sum_{j=1}^m\frac{U(x+e_i)}{m}-U(x+e_i)\right |\leq \tilde c|e| \ ,\ \forall (e,x)
$$
So we fix $\eta $ positive and pick $c = \tilde c^2\eta + 
\tilde c$. The above shows that
inequalities  (\ref{LP8}) and (\ref{LP6}) are satisfied.
With  Proposition \ref{Prop_Lyap}, we conclude that
 Property \textsf{TULES-NL} holds. Hence $e=0$ is (locally) exponentially stable manifold.
With inequalities (\ref{eq_EqNorme1}) and (\ref{eq_EqNorme2}) this implies that inequality (\ref{eq_Beta_exp}) holds.
\end{proof}
\par\vspace{1em}

In this result it is important to remark that there is no guarantee that the control law given here \startmodif ensures completeness \stopmodif of the solution.
Note however, that on the manifold $|w|_\DR=0$, the trajectories satisfy $\dot x = f(x)$ which is a complete system

\startmodif \noindent\textbf{Example 2 (continued):}
{\it
We come back to the example (\ref{eq_exemple}) in the case where 
$a=1$.
We note that the linear system (\ref{eq_exp}) is stabilizable by a feedback in the form
$$
u = -(1+2(\cos(x_0+t)) \begin{bmatrix}2&3\end{bmatrix}\de
$$}\stopmodif{\it
Indeed,  the solution of (\ref{eq_exp}) with the previous feedback satisfies
$$
\dot \de = (1+\cos(x_0+t))\begin{bmatrix}
0 & 1\\ -2 & -3
\end{bmatrix}
\de\ .
$$
Hence its solution are
$
\dE(e,x,t) = \psi(x,t)\de
$,
where $\psi$ is
 the generator of this time varying linear system given as
\begin{align*}
\psi(x,t) &= \exp\left ( (t+2\sin(x+t)\begin{bmatrix} 0 & 1\\-2&-3 \end{bmatrix}\right )\\
&= \varphi(x,t)^{-2}\begin{bmatrix}
-1+2\varphi(x,t) & -1 + \varphi(x,t)\\
2(1-\varphi(x,t)) &  2-\varphi(x,t)
\end{bmatrix}
\end{align*}
with $\varphi(x,t) = e^{t + 2\sin(t+x)}$.
}

{\it
Consequently, we get that $\de$ goes exponentially to zero. Hence Property $Q_1$ is satisfied.
}

{\it
We can then introduce the matrix $P$ solution to Q2  and given in (\ref{eq_P}) as
\begin{equation}\label{eq_P_exmple}
P(x) = \int_0^{+\infty}  \psi(x,s)^\prime \psi(x,s)ds \ .
\end{equation}
This matrix is positive definite and satisfies property Q2.
So we may want to use it for designing
an exponentially synchronizing control law.
With decomposing the $2\times 2$ matrix $P$ as
$$
P(w_i) = \begin{bmatrix}
P_{11}(w_{i2}) & P_{12}(w_{i2})\\P_{12}(w_{i2}) & P_{22}(w_{i2})
\end{bmatrix}\ ,
$$
we obtain
$
P(w)g(w) = \begin{bmatrix}
P_{12}(w_{i2})&P_{22}(w_{i2})
\end{bmatrix}^\prime
$.
Note that it can be shown (numerically) that
$$
P_{12}(w_{i2}) = \int_0^\infty \frac{4}{\varphi(w_{i2},t)^2} - \frac{9}{\varphi(w_{i2},t)^3} +\frac{5}{\varphi(w_{i2},t)^4} dt > 0\ .
$$
It follows the that function 
$\alpha(w) =  \frac{1}{P_{12}(w_{i2})}$
 is well defined and setting
$$
U(w_i) = w_{i1} + \int_0^{w_{i2}} \frac{P_{22}(s)}{P_{12}(s)}ds
$$
property (\ref{LP24}) is satisfied.
Hence, for this example, picking $\ell$ a sufficiently large real number,
the control law (\ref{eq_SyncConLaw}) ensures
local
exponential synchronization of $m$ agents.
We have checked this via simulation for the case
$m=5$, $\ell = 3$. 
The time evolution of the solution with $w_i(0)$, $i=1,\ldots, 5$ 
chosen randomly according to a uniform
distribution on $[0,10]$
is shown in
Figure \ref{fig:example1}a for $w_{i1}$ and \ref{fig:example1}b for 
$w_{i2}$.
}
\begin{figure}[h]
\centering
  \subfigure[The plot of $w_{i1}$, $i=1,\ldots 5$]{\includegraphics[height=1.8in]{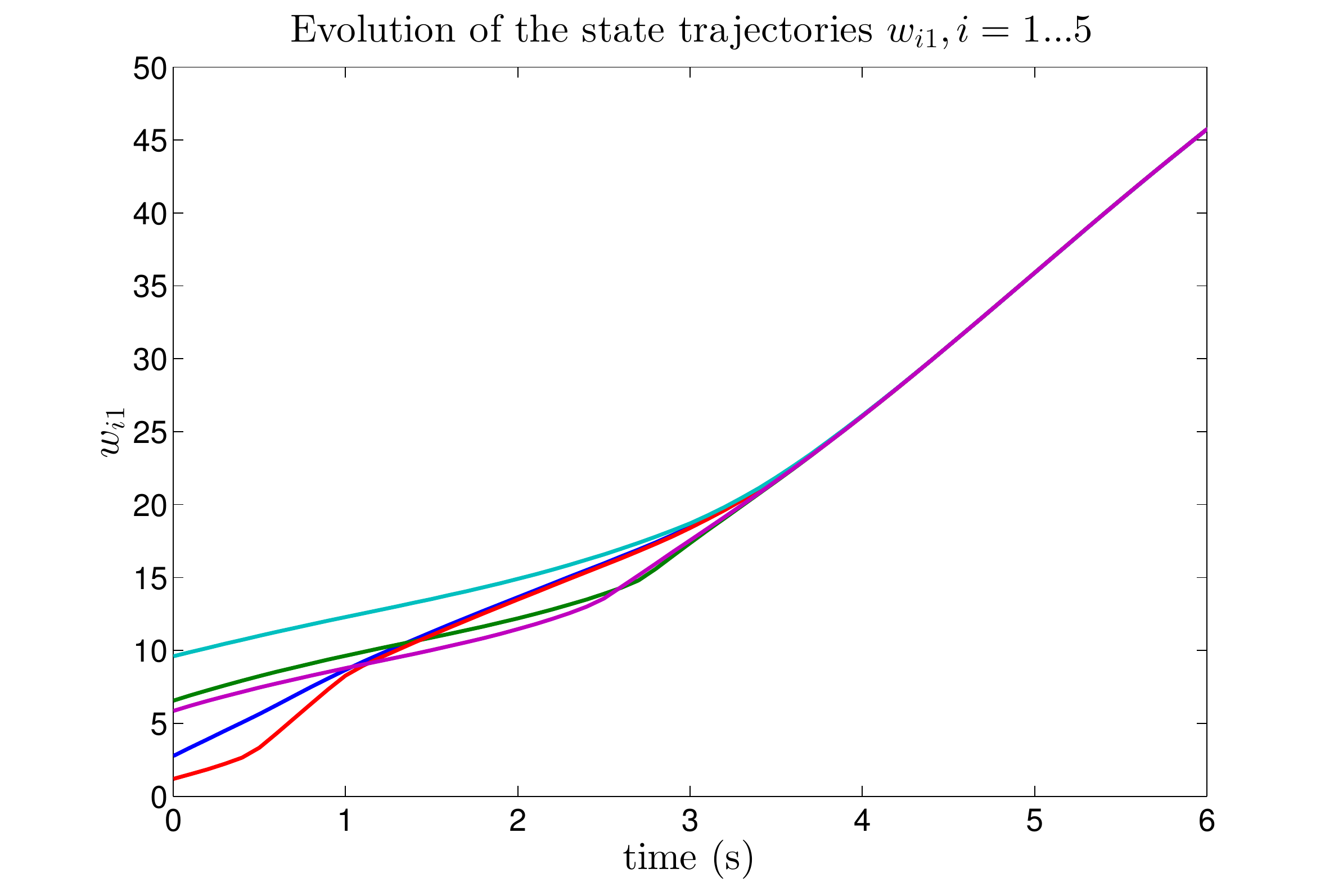}}
  \subfigure[The plot of $w_{i2}$, $i=1,\ldots 5$]{\includegraphics[height=1.8in]{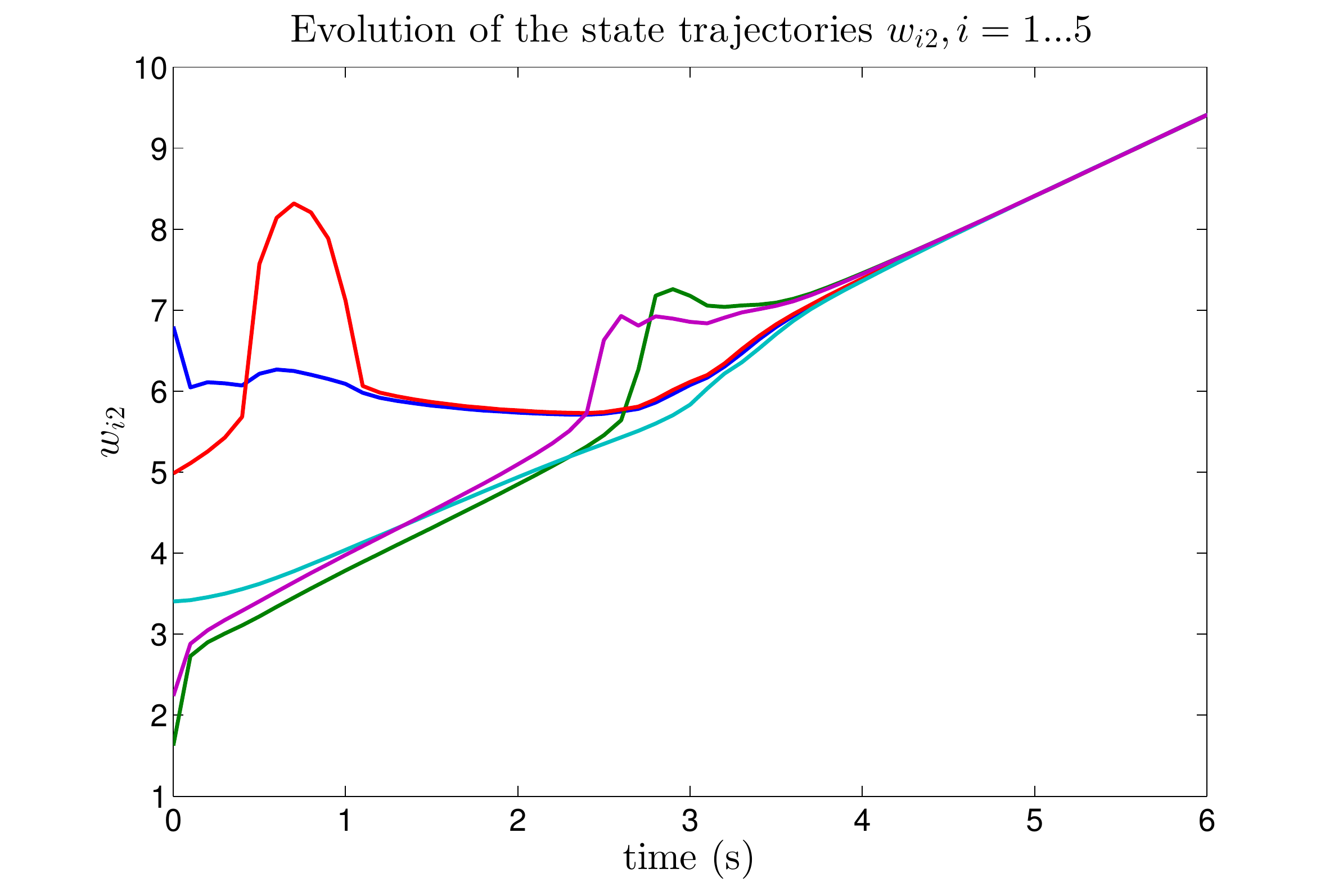}}
\caption{  Numerical simulation results of the
interconnected five agents as considered in Example 1.  The
left-figure shows the trajectories of the first state variable of each
agent, $w_{i1}, i=1\ldots 5$ while the right-figure shows those of the
second state variable of each agent $w_{i2}, i=1\ldots 5$.  The
simulation results show that the proposed contraction-based
distributed control law is able to synchronize the states $w_{i}$,
$i=1\ldots 5$.}
\label{fig:example1} 
\end{figure}
\par\vspace{1em}

As in the context of the observer design
given in \cite{SanfelicePraly_TAC_12}, 
a global result can be obtained
by imposing a further constraint on $P$. 
Specifically, the notion we need to introduce is the following.
\begin{definition}[Totally Geodesically Set]
Given 
a $C^2$ 
function $P$ defined on $\RR^n$ the values of which are symmetric 
positive definite matrices,
a $C^1$ function $\varphi:\RR^n\rightarrow\RR_+$ and a real number $\bar\varphi$, 
the  (level)  set $S = \{x\in\RR^n,\varphi(x)=\bar{\varphi}\}$   is 
said to be totally geodesic
with respect to $P$
 if, for any  $(x,v)$ in $S\times\RR^n$    such that  $$
\frac{\partial \varphi}{\partial x}(x)v =  0, v^\prime P(x)v=1\ ,
$$ 
any geodesic $\gamma$,
i.e. a solution of
\\[1em]$\displaystyle 
\frac{d}{ds}\left(\frac{d\gamma  ^*}{ds}(s)^\prime
P(\gamma ^*(s))\right)l 
$\refstepcounter{equation}\label{LP35}\hfill$(\theequation)$
\\\null \hfill $\displaystyle \;=\; 
\frac{1}{2}\frac{\partial }{\partial x}\left.\left(
\frac{d\gamma ^*}{ds}(s)^\prime
P(x)
\frac{d\gamma ^*}{ds}(s)\right)\right|_{x=\gamma^*(s)}
\  ,
$\\[1em]
with $\gamma(0)=x$ and $\frac{d\gamma}{ds}(0)=v$ satisfies
$$
\frac{\partial \varphi}{\partial x}(\gamma(s))\frac{d\gamma}{ds}(s) =  0\ , \forall s\ .
$$
\end{definition}

For the case of two agents only, we have the following.

\begin{proposition}[Global sufficient condition \startmodifOLD for $m=2$\stopmodifOLD]\label{Prop_GlobSynch}
Assume
\begin{list}{}{%
\parskip 0pt plus 0pt minus 0pt%
\topsep 0pt plus 0pt minus 0pt
\parsep 0pt plus 0pt minus 0pt%
\partopsep 0pt plus 0pt minus 0pt%
\itemsep 0pt plus 0pt minus 0pt
\settowidth{\labelwidth}{1.}%
\setlength{\labelsep}{0.5em}%
\setlength{\leftmargin}{\labelwidth}%
\addtolength{\leftmargin}{\labelsep}%
}
\item[1.]
there exist a  $C^3$  function $U:\RR^n\to \RR$ which
has bounded first and second derivatives,
and a $C^1$ function $\alpha :\RR^n\to \RR^p$ such
that, for all $x$ in $\RR^n$, (\ref{LP24}) is satisfied;
\item[2.]
there exist a positive real number $\lambda$, a   $C^3$   function $\PR :\RR^n\rightarrow\RR^{n\times n}$
and positive real numbers
$\underline{p}$ and $\overline{p}$,
such that  inequalities (\ref{LP10})
hold and we have,
for all $(x,v)$ in $\RR^n\times \RR^n$ such that $\frac{\partial U}{\partial x}(x)^\prime v=0$
\begin{equation}
\label{ArtsteinGlob}
\hskip -1em
\frac{1}{2}\,  v^\prime \der_f\PR (x)v +
v^\prime \PR (x)\frac{\partial f}{\partial x}(x)v
\leq  -
\,  \lambda \,  
 v^\prime P(x)v
\  ,
\end{equation}
\item [3.] For all $\bar U$ in $\RR$, the set $S=\{x\in\RR^n, U(x)=\bar U\}$ is totally geodesic with
respect to $P$.
\end{list}
Then
 there exists a function $\kell:\RR^{2n}\rightarrow\RR_+$, invariant by permutation 
such that, with the controls given by
$$
\phi_i(w)=\kell(w)\alpha (w_i)\,  (U(w_j)-U(w_i))
\  ,
$$
with $(i,j)\in \{(1,2), (2,1)\}$
the following holds and for all $w$ in $\RR^{2n}$,
\begin{equation}
|W(w,t)|_{\DR}\leq k|w|_\DR
\exp\left(-\frac{\lambda }{2}t\right)
\  ,
\end{equation}
where $t$ is any positive real number in the time
domain of definition of the closed loop solution.
\end{proposition}
\par\vspace{1em}
The  
proof of this result is given in Appendix 
\ref{Sec_ProofPropGlobSynch}. It
borrows some ideas of
\cite{SanfelicePraly_TAC_12}.
However,
different from \cite{SanfelicePraly_TAC_12}, we have here
a global convergence result. This follows
from the fact that in the high gain parameter $\kell$, the norm of the full state space can be 
used
 (and not
only the norm of the estimate
as in the observer case
).

Note that nothing is said about the
domain
 of existence of the solution.

\section{Conclusion}
\label{Sec_Conclusion}
We have studied the relationship between the exponential stability of an invariant manifold and the existence of a Riemannian metric for which the flow is ``transversally'' contracting.
It was shown that the following properties are  equivalent  
\begin{enumerate}
\item A manifold is ``transversally'' exponentially stable;
\item The ``transverse'' linearization along any solution in the manifold is exponentially stable;
\item There exists
a field of positive definite quadratic forms whose
 restrictions to the transverse direction to the
manifold are decreasing along the flow.
\end{enumerate}
\startmodifbj As an illustrative example for these equivalence results, we have revisited the property of exponential incremental stability where we can obtain a global result. 

The characterization of transverse exponential stability has allowed us to investigate a necessary condition for two different control problems of nonlinear observer design and of synchronization of nonlinear multi-agent systems which leads to a novel constructive design for each problem. Recent result by others has also shown the applicability of our results beyond these two control problems. \stopmodifbj 
Although the main results hold for local uniform transverse exponential stability,
we show that global results can also be obtained in some particular cases. The extension of all the results to the global case is currently under study. 

\section*{Acknowledgement}
 The authors would like to thank the anonymous reviewers for their helpful comments.

\appendix

\subsection{Proof of Proposition \ref{Prop_DetecNec}}
\label{SecProof_Prop_DetecNec}
\begin{proof}
Let us start with some estimations.
Let $z=e-\de$. Along solutions of
(\ref{eq_System}) and (\ref{eq_System_dif}), we have
$$
\dot z = F(e,x)-\frac{\partial F}{\partial e}(0,\dx)\de =
\frac{\partial F}{\partial e}(0,\dx)z + \Delta(x,e,\dx)
$$
with the notation
\begin{eqnarray*}
\Delta(e,x,\dx)
&\hskip -0.3em =& \hskip -0.3em F(e,x)-\frac{\partial F}{\partial e}(0,\dx)e
\  ,
\\[0.1em]
&\hskip -0.3em =&\hskip -0.3em \left[ F(e,x)-F(e,\dx) \right]
\\&\hskip -0.3em &   
+\; \left[ F(e,\dx)- F(0,\dx)-\frac{\partial F}{\partial e}(0,\dx)e\right]\ .
\end{eqnarray*}
Note that the manifold $\mathcal{E}:=\{(e,x):\,  e=0\}$ being invariant, it yields
\begin{equation}
\label{LP3}
F(0,x) = 0\qquad \forall x
\  .
\end{equation}
With Hadamard's lemma, (\ref{LP3}) and (\ref{LP2}), we obtain the existence of
positive real numbers $c_1$ and $c_2$ such that,
for all $(e,x,\dx)$ in $B_e(kr)\times\RR^{n_x}\times\RR^{n_x}$,
$$
|\Delta(e,x,\dx)|\leq c_1 |e|^2+ c_2 |e||x-\dx|
\ .
$$
This, with (\ref{LP5}), gives,
for all $(e,\de,x,\dx)$ in
$B_e(kr)\times\RR^{n_e}\times\RR^{n_x}\times\RR^{n_x}$,
\footnote{
\label{foonote1}
Here the notation $\dot{\overparen{|z|}}$ is 
abusive. The function $x\mapsto |x|$ is not $C^1$ but only 
Lipschitz. Nevertheless given a vector field $f$ an upper right Dini Lie derivative, i.e.
$
\limsup_{h\to 0_+} \frac{|x+h f(x)|-|x|}{h}
$
does exist and, by the triangle inequality, we have
$$
-|f(x)|\leq \limsup_{h\to 0_+} \frac{|x+h f(x)|-|x|}{h}\leq |f(x)|\  .
$$
So here and in the following $\dot{\overparen{|x|}}$ denotes this 
upper right Dini Lie derivative.
} 
\begin{equation}\label{eq_dot_z_upperbound}
\dot{\overparen{|z|}}\leq \mu|z|+c_1|e|^2 + c_2 |e||x-\dx|
\ .
\end{equation}
Similarly (\ref{eq_System}), (\ref{eq_System_dif}) and (\ref{LP2})
give,
for all $(e,x,\dx)$ in $B_e(kr)\times\RR^{n_x}\times\RR^{n_x}$,
\\[0.7em]$\displaystyle
\dot{\overparen{|x-\dx| }}
$\hfill \null
\\[0.3em]\null\hfill$
\begin{array}[b]{@{}cl@{}}
\leq &\displaystyle
|G(e,x)-G(0,x)|\;+\; |G(0,x)-G(0,\dx)|
\  ,
\\[0.3em]
\leq & \displaystyle c|e|+ \rho  |x-\dx|
\ .
\end{array}
$\refstepcounter{equation}\label{eq_dot_x_upperbound}\hfill$(\theequation)$
\\

Now let $\dr$ be a positive real number smaller than $r$
and $S$ be a 
positive real number both to be made precise later on. Let $\de_0$ in $B_e(\dr)$ and $\dx_0$ in $\RR^{n_x}$  be arbitrary
and let $(\dE(\de_0,\dx_0,t),\dX(\dx_0,t))$ be the corresponding
solution of (\ref{eq_System_dif}). Because  of the completeness
assumption on (\ref{eq_System}),
the linearity of (\ref{eq_System_dif}) and the fact that
$(0,\dX(\dx_0,t))$ is solution of both (\ref{eq_System}) and 
(\ref{eq_System_dif}), 
$(\dE,\dX)$ is defined on $[0,+\infty )$.
We denote~:
$$
\de_i=\dE(\de_0,\dx_0,iS)\  ,\quad
\dx_i=\dX(\dx_0,iS)
\qquad \forall i\in
\NN\  \footnote{$\NN$ denotes the set of integers.}
$$
and consider the corresponding solutions
$(E(\de_i,\dx_i,s),X(\de_i,\dx_i,s))$ of (\ref{eq_System}). By
assumption, they are defined on $[0,+\infty )$ and,
because of (\ref{eq_ExpStab}), if $\de_i$ is in $B_e(r)$,
then $E(\de_i,\dx_i,s)$ is in
$B_e(kr)$ for all positive times $s$, making possible the use of
inequalities (\ref{eq_dot_z_upperbound}) and (\ref{eq_dot_x_upperbound}). Finally, for each
integer $i$,
we define the following time functions on $[0,S]$
\begin{eqnarray*}
Z_i(s) &= &|E(\de_i,\dx_i,s) - \dE(\de_0,\dx_0,s+iS)|\ ,
\\
W_i(s) &= &
|X(\de_i,\dx_i,s)-\dX(\dx_0,s+iS)|
\ .
\end{eqnarray*}

Note that we have $Z_i(0)=W_i(0)=0$.

From the inequalities (\ref{eq_dot_x_upperbound}), and
(\ref{eq_ExpStabDrift}), we get,
for each integer $i$
such that $\de_i$ is in $B_e(r)$,
and for all $s$ in $[0,S]$,
\begin{eqnarray*}
W_i(s) &\leq &c\int_0^s\exp( \rho(s-\sigma ))|E(\de_i,\dx_i,s)|d\sigma
\ ,
\\
&\leq &c\int_0^s\exp( \rho(s-\sigma ))k\exp(-\lambda \sigma )d\sigma |\de_i|\ ,
\\
&\leq &ck\exp( -\lambda  s)\frac{\exp( (\rho+\lambda) s)-1}{\rho+\lambda}|\de_i|
\ .
\end{eqnarray*}
Similarly, using (\ref{eq_dot_z_upperbound}) and Gr\"onwall inequality we get
\\[0.7em]$\displaystyle
Z_i(s)$\hfill \null \\$\displaystyle 
\leq \:    c\!\int_0^s\!\!\exp(\mu (s
-\sigma ) ) 
|E(\de_i,\dx_i,\sigma )|
\left(|E(\de_i,\dx_i,\sigma 
)| \!+
W_i(\sigma )\right)
d\sigma
$\\[0.7em]$
\leq \:  \startmodifOLD \gamma(s)\,  |\de_i|^2 \stopmodifOLD
\qquad
\forall s\in [0,S]
\  ,
$\\[0.7em]
where we have used the notation,
\\[0.7em]\vbox{\noindent$\displaystyle
\gamma(s) = \startmodifOLD c\int_0^s\exp(\mu (s
-\sigma ) ) k\exp(-2\lambda \sigma )
\times \stopmodifOLD
$\\[0.3em]\null\hfill$\displaystyle
\times
\left(k + ck\exp( -\lambda  \sigma )\frac{\exp( (\rho+\lambda) \sigma )-1}{\rho+\lambda}\right)d\sigma
\  .
$}\\[0.7em]
With all this, we have obtained that, if we have $\de_j$ in $B_e(r)$  for all $j$ in
$\{0,\ldots,i\}$, then we have also,
for all $s$ in  $[0,S]$ and all $j$ in $\{0,\ldots,i\}$,
\begin{eqnarray*}
| \dE(\de_0,\dx_0,s+jS)| &=& | \dE(\de_j,\dx_j,s)|
\  ,
\\
&\leq &|E(\de_j,\dx_j,s)| + |Z_j(s)|
\  ,
\\
&\leq  &
\left[
\vrule height 0.5em depth 0.5em width 0em
k\exp(-\lambda s) +\gamma(s)|\de_j|\right]|\de_j|
\end{eqnarray*}
Now, given a real number $\varepsilon $ in $(0,1)$, we select $S$ and $\dr$ to satisfy~:
\begin{eqnarray*}
&\displaystyle k\exp(-\lambda S) \leq \frac{\min\{k,1-\varepsilon \}}{2}
\  ,
\\
&\displaystyle
\dr\leq \min\left\{r,\frac{\min\{k,1-\varepsilon \}}{2\sup_{s\in [0,S]}
\gamma (s)}\right\}
\  .
\end{eqnarray*}
Then, for all $\de_j$ smaller in norm than $\dr$, we have
$$
|\dE(\de_0,\dx_0,s+jS)|\leq (1-\epsilon)\,   |\de_j|
$$
So, since
$\de_0$ is in $B_e(\dr)$, it follows by induction that we have~:
$$
|\de_i|=| \dE(\de_0,\dx_0,iS)| \leq (1-\varepsilon )^i \,  \dr\leq \dr
\qquad \forall i \in \NN
\  .
$$
Since,
with (\ref{LP5}),
we have also
$
\dot{\overparen{|\de|}}
\leq \mu |\de|
$,
we have established,
for all $s$ in $[0,S]$ and all $i$ in $\NN$,
$$
|\dE(\de_0,\dx_0,s+iS)|\leq \exp(\mu s)(1-\epsilon)^i |\de_0|
$$
and therefore,
for all $(\de_0,\dx_0,t)$ in
$B_e(a)\times\RR^{n_x}\times \RR_{\geq 0}$,
$$
|\dE(\de_0,\dx_0,t)|\leq \exp(\mu S)(1-\epsilon)^\frac{t-S}{S} |\de_0|
\  .
$$
By rearranging this inequality and taking advantage of the
homogeneity of the system (\ref{eq_System_dif}) in the $\de$ component,
we have obtained (\ref{eq_ExpStabDrift}) with $\tilde
k=\frac{\exp(\mu S)}{1-\epsilon}$ and $\tilde\lambda = \startmodifOLD - \frac{\ln(1-\epsilon)}{S} \stopmodifOLD $.
\end{proof}

\subsection{Proof of Proposition \ref{Prop_ExistTensor}}
\label{Sec_ProofPropTensor}
\begin{proof}
Let $(\dE(\de_0,\dx_0,t),\dX(\dx_0,t))$ be the solution
of (\ref{eq_System_dif}) passing through an arbitrary pair
$(\de_0,\dx_0)$ in  $\RR^{n_e}\times\RR^{n_x}$. By assumption, it is defined on
$[0,+\infty )$.

For any $v$ in $\RR^{n_e}$, we have
$$
\frac{\partial }{\partial t}
\left(\frac{\partial \dE}{\partial \de}(0,\dx_0,t)v \right)
= \frac{\partial F}{\partial e}(0,\dX(\dx_0,t))\frac{\partial \dE}{\partial \de}(0,\dx_0,t) v
\  .
$$
Uniqueness of solutions then implies,
for all $(\de_0,\dx_0,t)$ in
$\RR^{n_e}\times\RR^{n_x}\times\RR_{\geq 0}$,
$
\dE(\de_0,\dx_0,t)=\frac{\partial \dE}{\partial \de}(0,\dx_0,t)\de_0
$
and our assumption (\ref{eq_ExpStabDrift}) gives,
for all $(\de_0,\dx_0,t)$ in
$\RR^{n_e}\times\RR^{n_x}\times\RR_{\geq 0}$,
$$
\left|\frac{\partial \dE}{\partial \de}(0,\dx_0,t)\de_0\right|\leq \dk |\de_0|\exp(-\tilde \lambda t)
$$
and therefore
$$
\left|\frac{\partial \dE}{\partial \de}(0,\dx_0,t)\right|\leq \dk \exp(-\tilde \lambda t)
\qquad \forall (\dx_0,t)\in
\RR^{n_x}\times\RR_{\geq 0}
\ .
$$
This allows us to claim that, for every symmetric positive definite
matrix $Q$, the function $\PR :\RR^{n_x}\rightarrow\RR^{n_e\times n_e}$ given by \startmodifOLD (\ref{eq_P}) \stopmodifOLD
is well defined, continuous and satisfies
$$
\lambda_{\max}\{\PR (\dx)\}\leq \frac{\dk ^2}{2\tilde \lambda }\lambda_{\max}\{Q\}=\overline{p}
\qquad \forall \dx\in \RR^{n_x}
\  .
$$
Moreover we have
\\[0.7em]$\displaystyle
\frac{\partial }{\partial t}
\left(
v^\prime \left[\frac{\partial \dE}{\partial \de}(0,\dx_0,t)\right]^{-1}
\right)
$\hfill \null\\[0.3em]\null\hfill$\displaystyle
=
-v^\prime  \left[\frac{\partial \dE}{\partial \de}(0,\dx_0,t)\right]^{-1}
\frac{\partial F}{\partial e}(0,\dX(\dx_0,t))
\  ,
$\\[1em]
With (\ref{LP5}), this yields
$$
\left|v^\prime \left[\frac{\partial \dE}{\partial 
\de}(0,\dx_0,t)\right]^{-1}\right|
\leq \exp(\mu t)
\left|v\right|
$$
and implies
\\[1em]$\displaystyle 
[v^\prime v]^2
$\hfill \null \\\null \hfill $\displaystyle 
\begin{array}{cl}
\leq &\displaystyle 
\left|v^\prime\left[\frac{\partial \dE}{\partial 
\de}(0,\dx_0,t)\right]^{-1}\right|^2
 \left|\frac{\partial \dE}{\partial \de}(0,\dx_0,t)v \right|^2
\qquad \qquad \quad 
\\[0.7em]
\leq &\displaystyle 
\frac{1}{\lambda_{\min}\{Q\}}
\left|v^\prime\left[\frac{\partial \dE}{\partial \de}(0,\dx_0,t)\right]^{-1}\right|^2
\times\\&\multicolumn{1}{r@{}}{ \displaystyle %
\times\; 
v^\prime\frac{\partial \dE}{\partial \de}(0,\dx_0,t)^\prime Q \frac{\partial \dE}{\partial \de}(0,\dx_0,t)v 
}
\\
\\[0.7em]
\leq &\displaystyle 
\frac{|v|^2 \exp(2\mu t)}{\lambda_{\min}\{Q\}}\: 
v^\prime\frac{\partial \dE}{\partial \de}(0,\dx_0,t)^\prime Q \frac{\partial \dE}{\partial \de}(0,\dx_0,t)v 
\end{array}$\\[1em]
This gives
$$
\underline{p}= \frac{1}{2\mu  }\lambda_{\min}\{Q\}
\leq \lambda_{\min}\{\PR (\dx)\}
\qquad \forall \dx\in \RR^{n_x}
\  .
$$

Finally, to get (\ref{eq_TensorDerivative}), let us exploit the semi
group property of the solutions.
We have for all $(\de,\dx)$ in $\RR^{n_x}\times\RR^{n_e}$ and all
$(t,r)$ in $\RR_{\geq 0}^2$
$$
\dE(\dE(\de,\dx,t),\dX(\dx,t),r)=
\dE(\de,\dx, t+r)\ .
$$
Differentiating with respect to $\de$ the previous equality yields
$$
\frac{\partial \dE}{\partial \de}(\dE(\de,\dx,t),\dX(\dx,t),r)
\frac{\partial \dE}{\partial \de}(\de,\dx,t)
=
\frac{\partial \dE}{\partial \de}(\de,\dx, t+r)
$$
Setting in the previous equality
$$
(\de,\dx):= (0,\dX(\dx,h))\ ,\ h:=-t\ ,\ s:=t+r
$$
we get for all $\dx$ in $\RR^{n_x}$ and all $(s,h)$ in $\RR^2$
$$
\frac{\partial \dE}{\partial \de}(0,\dx,s+h)
\frac{\partial \dE}{\partial \de}(0,\dX(\dx,h),-h)
=
\frac{\partial \dE}{\partial \de}(0,\dX(\dx,h), s)\ .
$$

Consequently, this yields,
\\[1em]\vbox{\noindent$
\PR (\dX(\dx,h))
$\hfill \null\\[0.3em]\null\hfill$
\begin{array}{@{}l@{}}
\displaystyle
= \lim_{T\to +\infty }\int_0^{T} \left(
\frac{\partial \dE}{\partial e}(0,\dX(\dx,h),s)\right)^\prime Q\frac{\partial \dE}{\partial \de}(0,\dX(\dx,h),s)
ds
\\
\displaystyle
= \lim_{T\to +\infty }
\left(
\frac{\partial \dE}{\partial \de}(0,\dX(\dx,h),-h)
\right)^\prime  \times
\\\multicolumn{1}{@{}c@{}}{%
\displaystyle
\times
\left[\int_0^{T} \left(
\frac{\partial \dE}{\partial \de}(0,\dx,s+h)
\right)^\prime Q\frac{\partial \dE}{\partial \de}(0,\dx,s+h)
ds
\right]
\times
}
\\\multicolumn{1}{@{}r@{}}{%
\displaystyle
\times
\frac{\partial \dE}{\partial \de}(0,\dX(\dx,h),-h)
}
\end{array}
$}\\[1em]
But we have~:
\begin{eqnarray*}
&\displaystyle
\lim_{h\to 0}\frac{\frac{\partial \dE}{\partial \de}(0,\dX(\dx,h),-h)-I}{h}
= -\frac{\partial F}{\partial e}(0,\dx)
\  ,
\\[0.5em]
&\displaystyle
\lim_{h\to 0}\frac{
\frac{\partial \dE}{\partial \de}(0,\dx,s+h)
-
\frac{\partial \dE}{\partial \de}(0,\dx,s)
}{h}
=
\frac{\partial }{\partial s}\left(\frac{\partial \dE}{\partial \de}(0,\dx,s)\right)
\end{eqnarray*}
and
\\[0.7em]$\displaystyle
\int_0^T
\frac{\partial }{\partial s}\left(\frac{\partial \dE}{\partial
\de}(0,\dx,s)\right)^\prime
Q
\left(\frac{\partial \dE}{\partial \de}(0,\dx,s)\right)
ds
$\hfill \null\\[0.3em]\null\hfill$\displaystyle
+
\int_0^T
\left(\frac{\partial \dE}{\partial
\de}(0,\dx,s)\right)^\prime
Q
\frac{\partial }{\partial s}\left(\frac{\partial \dE}{\partial \de}(0,\dx,s)\right)
ds
$\hfill \null\\[0.3em]\null\hfill$\displaystyle
=
\left(\frac{\partial \dE}{\partial \de}(0,\dx,T)\right)^\prime
Q
\left(\frac{\partial \dE}{\partial \de}(0,\dx,T)\right)
-
Q
\  .$\\[0.7em]
Since $\lim_T$ and $\lim_h$ commute because of the exponential
convergence to $0$ of $\frac{\partial \dE}{\partial
\de}(0,\dx,s)$, we conclude that the derivative (\ref{LP9}) does
exist and satisfies (\ref{eq_TensorDerivative}).
\end{proof}
\subsection{Proof of Proposition \ref{Prop_Lyap}}
\label{Sec_ProofPropLyap}
\begin{proof}
Consider the function $V(e,x) = e^\prime \PR (x)e$.
Using (\ref{eq_TensorDerivative}), the time derivative of $V$
along the solutions of the system (\ref{eq_System}) is given, for all
$(e,x)$, by
\\[0.7em]$\displaystyle
\dot {\overparen{V(e,x)}} 
\; = \;
2e^\prime P(x)F(e,x) + \frac{\partial e^\prime \PR (\cdot)e}{\partial x} (x)G(e,x)$\\[0.5em]$\displaystyle \null\hphantom{\dot {\overparen{V(e,x)}} }\;\leq
- e^\prime Qe +  2e^\prime \PR (x)  \left[F(e,x)-\frac{\partial F}{\partial e}(0,x)e\right]
$\hfill \null\\[0.7em]\null\hfill$\displaystyle
+  \frac{\partial e^\prime \PR (\cdot)e}{\partial x} (x)\left[G(e,x)-G(0,x)\right]
\; .
$\\[1em]
On the other hand, using Hadamard's Lemma and (\ref{LP6}), we get~:
\begin{eqnarray*}
&\displaystyle \left|F(e,x)-\frac{\partial F}{\partial e}(0,x)e\right|
\leq c |e|^2
\  ,
\\
&\displaystyle
\left|G(e,x)-G(0,x)\right| \leq  c |e|
\qquad \forall (e,x)\in B_e(\eta ) \times \RR^{n_x}
\  .
\end{eqnarray*}
These inequalities together with (\ref{LP7}) and (\ref{LP8}) imply,
for all $(e,x)$ in $B_e(\eta ) \times \RR^{n_x}$,
$$
\dot {\overparen{V(e,x)}} \leq
- \left[
\frac{\lambda _{\min}\{Q\}}{\overline{p}}
-2c(1+c)\frac{\overline{p}}{\underline{p}} |e|\right]V(e,x)
\  .
$$
 It shows immediately  that (\ref{eq_ExpStab}) holds with $r$, $k$ and $\lambda $ satisfying~:
\begin{eqnarray*}
r  &<&\frac{\null\:   \underline{p}\: \null   }{\overline{p}}
\min\left\{\eta
,\frac{\lambda _{\min}\{Q\}}{2\overline{p}c(1+c)}
\right\}
\; ,
\\
k&=& \startmodifOLD \sqrt{\frac{\null \,  \overline{p}\,  \null }{\underline{p}}}\stopmodifOLD
\; ,
\\
\lambda &=&\startmodifOLD
\left[
\frac{\lambda _{\min}\{Q\}}{2\overline{p}}
-r c(1+c)\frac{\,  \overline{p}\,  }{\underline{p}} \right]\stopmodifOLD
\;  .
\end{eqnarray*}
\end{proof}

\subsection{Proof of Proposition \ref{Prop_GlobSynch}}
\label{Sec_ProofPropGlobSynch}

The result holds when $w$ is in $\DR$ or when $U$ is constant 
(since (\ref{ArtsteinGlob}) holds for all $v$). So,
in view of \cite[Proposition A.2.1]{SanfelicePraly_TAC_12}, we can 
assume without loss of generality that $\frac{\partial U}{\partial 
w}$ has nowhere a zero norm and, in the following, we 
restrict our attention to $\RR^{2n}\!\setminus\!\DR$.
In $\RR^{2n}\!\setminus\!\mathcal{D}$ the dynamics of $w$ is
$$
\dot w_i = f(w_i) \;+\;\kell(w)g(w_i)\alpha (w_i)\,  \sum_{j=1}^2\left [  U(w_j) -U(w_i)\right ]
$$
With the $C^2$ matrix function $P$  we define the Riemannian
length of
a piece-wise $C^1$ path $\gamma :[s_1,s_2]\to
\RR^n$, between $w_1=\gamma (s_1)$ and $w_2=\gamma (s_2)$  as in (\ref{eq_RiemanianLength}) and the corresponding 
distance $d(w_1,w_2)$ by minimizing along all such path

Because of (\ref{LP10}) and the fact that $P$ is  \startmodifOLD $C^3$, \stopmodifOLD 
Hopf-Rinow Theorem implies
the metric space we obtain
this way
is complete, and,
given any  $w_1$ in $\RR^{n}$ and $w_2$ in $\RR^n$, there exists a  \startmodifOLD $C^3$ \stopmodifOLD 
normalized\footnote{%
This means that $\gamma ^*$ satisfies
\null \quad $\displaystyle
\frac{d\gamma^*}{ds}(s)^\prime  P (\gamma^*(s))\frac{d\gamma^*}{ds}(s)=1\ .
$\hfill \null
}
 minimal geodesic
$\gamma^*$,
solution of (\ref{LP35}),
 such that
\begin{eqnarray}
\nonumber&\displaystyle
w_1=\gamma ^*(s_1)
\  ,\quad
w_2=\gamma ^*(s_2)
\  ,
\\[0.3em]
&\displaystyle\label{LyapSynch}
d(w_1,w_2) = \left. L(\gamma^*)\vrule height 0.51em depth 0.51em
width 0em \right|_{s_1}^{s_2}=s_2-s_1
\ .
\end{eqnarray}

Following  \cite{SanfelicePraly_TAC_12},  for  each $s$ in $[s_1,s_2]$ consider the $C^1$ function
$t\mapsto \Gamma(s,t)$ solution of
\\[0.7em]$\displaystyle
\displaystyle
\frac{\partial \Gamma}{\partial t}(s,t) =
f(\Gamma(s,t))
$\hfill \null \\[0.3em]$\displaystyle
\null \quad +
\kell(W(w,t))\,  
g(\Gamma(s,t))\alpha (\Gamma(s,t))
\times
$\hfill \null \\[0.3em]\null \hfill $\displaystyle
\times
\sum_{j=1}^2 \left [U(W_j(w,t))-U(\Gamma(s,t) )\right ]
$\\[0.7em]
with initial condition
\begin{equation}
\label{LP25}
\Gamma(s,0) = \gamma^*(s)
\; .
\end{equation}
With (\ref{LP22}), we have
$$
\Gamma (s_1,t)=W_1(w,t)
\  ,\quad
\Gamma (s_2,t)=W_2(w,t)
$$
and so, for each $t$, $s\in [s_1,s_2]\mapsto \Gamma (s,t)$ is
a $C^2$ path between $W_1(w,t)$ and $W_2(w,t)$.
From the first variation formula
(see \cite[Theorem 6.14]{Spivak_Book_79} for instance\footnote{In \cite[Theorem 6.14]{Spivak_Book_79}, the result is stated with $\gamma^*$ $C^\infty$ note however that $C^2$ is enough.}), 
we have
$$
\ddt \left.\left(\left.L(\Gamma(s,t))
\vrule height 0.51em depth 0.51em width 0em
\right|_{s_1}^{s_2}\right)\vrule height 0.51em depth 0.51em width 0em
\right|_{t=0}\;=\;
\startmodifOLD
\bida (w) \;+\; \bidc(w)
\stopmodifOLD
$$
where
\\[1em]$\displaystyle 
\bida (w)\;=\; \kell(w)\left [U(w_1) -U(w_2)\right ]\frac{d\gamma ^*}{ds}(s_2)^\prime
\startmodifOLD
P(w_2)g(w_2)\alpha (w_2)
\stopmodifOLD
$\\\hfill \null \\\null \hfill$\displaystyle 
 -\kell(w) \left [U(w_2) -U(w_1)\right ]\frac{d\gamma ^*}{ds}(s_1)^\prime
P(w_1)g(w_1)\alpha (w_1)
$\\[1em]$\displaystyle 
\bidc (w)\;=\; \frac{d\gamma ^*}{ds}(s_2)^\prime
P(w_2)f(w_2)
\;-\; 
\frac{d\gamma ^*}{ds}(s_1)^\prime
P(w_1)f(w_1)
$\\[1em]
But, with (\ref{LP24}), we obtain~:
\\[1em]$\displaystyle 
\startmodifOLD
\bida 
\stopmodifOLD
(w)\;=\; 
 -\kell(w)\left[U(w_2)-U(w_1)\right]\times
$\\\hfill \null \\\null \hfill $\displaystyle 
\times
\left[
\frac{\partial U}{\partial x}(w_2)\frac{d\gamma ^*}{ds}(s_2)
+ 
\frac{\partial U}{\partial x}(w_1)\frac{d\gamma ^*}{ds}(s_1)
\right]
$\\[1em]
Also, with the Euler-Lagrange form of the geodesics equation
(\ref{LP35}),
we get~:
\\[1em]\vbox{\noindent$\displaystyle 
\startmodifOLD
\bidc
\stopmodifOLD
 (w)
$\\\hfill \null \\\null \hfill $\displaystyle 
\begin{array}{@{}c@{\; }l@{}}
=&\displaystyle \int_{s_1}^{s_2}
\left[\frac{d}{ds}\left(\frac{d\gamma ^*}{ds}(s)^\prime
P(\gamma ^*(s))\right)f(\gamma ^*(s))
\right.
\\
\multicolumn{2}{@{}r@{}}{%
\left.
\displaystyle \;+\; 
\frac{d\gamma ^*}{ds}(s)^\prime
P(\gamma ^*(s))\frac{d}{ds}\left(
\vrule height 0.5em depth 0.5em width 0pt
f(\gamma ^*(s))\right)\right]ds
}
\  ,
\\[0.5em]
=&\displaystyle 
\int_{s_1}^{s_2}
\left[
\frac{1}{2}\frac{\partial }{\partial x}\left.\left(
\frac{d\gamma ^*}{ds}(s)^\prime
P(x)
\frac{d\gamma ^*}{ds}(s)\right)\right|_{x=\gamma^*(s)}
\!\!f(\gamma ^*(s))
\right.
\qquad 
\\
\multicolumn{2}{@{}r@{}}{%
\left.
\displaystyle \;+\; 
\frac{d\gamma ^*}{ds}(s)^\prime
P(\gamma ^*(s))\frac{\partial f}{\partial x}(\gamma ^*(s))\frac{d\gamma ^*}{ds}(s)\right]ds
}
\  .
\end{array}
$}\\[1em]
Here the integrand is nothing but the left hand side of 
(\ref{ArtsteinGlob}). With a compactness argument\footnote{%
The following two properties are equivalent\\
a) $ v^\prime f(x) v < 0$ for all $v$ with $|v|=1$ and all $x$ 
satisfying $g(x)v=0$ \\
b) there exists $\nu $ such that
$v^\prime f(x) v - \nu (x) |g(x)v|^2 \leq 0$ for all $v$ with $|v|=1$ 
and all $x$.
\\
\textit{Proof} b) $\Rightarrow$ a) is trivial. For the converse, let 
$C$ be an arbitrary compact set, if b) does not hold for
some $\eta _C$ and all
$x$ in $C$, 
there exist $x_i$ and $v_i$ with $|v_i|=1$ satisfying
$v_i^\prime f(x_i) v _i \geq i |g(x_i)v_i|^2 $. With compactness 
this implies the existence of $x_\omega $ and $v_\omega$ with $|v_\omega|=1$ satisfying
$g(x_\omega)v_\omega=0$ and $v_\omega ^\prime f(x_\omega ) v _\omega 
\geq 0$. This contradicts a).
}
we can show that 
condition (\ref{ArtsteinGlob}) in Proposition \ref{Prop_GlobSynch} is equivalent to the existence
of a smooth function $\nu  :\RR^n\to \RR_+$ such that, for all 
$(x,v)$,
\\[1em]$\displaystyle 
\frac{1}{2}\,  v^\prime \der_f\PR (x)v +
v^\prime \PR (x)\frac{\partial f}{\partial x}(x)v
$\hfill \null \\\null \hfill $\displaystyle 
\leq  \; -  \lambda \,   v^\prime P(x)v
\;+\; 
\nu  (x)\left|\frac{\partial U}{\partial x}(x)v\right|^2
\  .
$\\[1em]
Hence, the geodesic being normalized,  we have~:
\\[1em]$\displaystyle 
\bidc (w) \;+\; \lambda \,  \int_{s_1}^{s_2}
\frac{d\gamma ^*}{ds}(s)^\prime P(\gamma ^*(s)
\frac{d\gamma ^*}{ds}(s) ds
$\hfill \null \\[0.5em]\null \hfill $ \displaystyle 
=\; \bidc (w) \;+\; \lambda \,  d(w_1,w_2)
\; \leq\; \bidb (w) 
\  ,
$\\[0.7em]
with the notation~:
$$
\bidb (w) \;=\; \int_{s_1}^{s_2}
\nu  (\gamma^*(s))
\left|\frac{\partial U}{\partial x}(\gamma ^*(s))\frac{d\gamma ^*}{ds}(s)\right|^2 ds
\ . 
$$
From $\bida$ and $\bidb$ we define two   $C^2$   functions $a$ and $b$ 
by dividing by $d(w_1,w_2)=s_2-s_1$. Namely, we define~:
\\[1em]$\displaystyle 
a _{\gamma ^*}(w,r)\;=\; \frac{U(\gamma ^*(r+s_1))-U(w_1)}{r}
\; \times
$\refstepcounter{equation}\label{LP31}\hfill$(\theequation)$
\\[0.7em]\null \hfill $\displaystyle 
\times\; 
\left[
\frac{\partial U}{\partial x}(\gamma ^*(r+s_1))
\frac{d\gamma ^*}{ds}(\gamma ^*(r+s_1))
+ 
\frac{\partial U}{\partial x}(w_1)\frac{d\gamma ^*}{ds}(s_1)
\right]
$\\[1em]
\begin{equation}
\label{LP33}
b_{\gamma ^*}(w,r) = \frac{1}{r}
\int_{s_1}^{r+s_1}
\hskip -1em\nu  (\gamma^*(s))
\left|
\frac{\partial U}{\partial x}(\gamma ^*(s))\frac{d\gamma ^*}{ds}(s)
\right|^2 \! ds
\,  .
\end{equation}
They are defined on
$\RR^{2n}\!\setminus\!\mathcal{D} \: \times \; ]0,d(w_1,w_2)]$
and depend a priori on 
the particular minimizing geodesic $\gamma ^*$ we consider. We extend 
by continuity (in  $r$  ) their definition to
$\RR^{2n}\!\setminus\!\mathcal{D} \: \times \; [0,d(w_1,w_2)]$
by letting
\begin{align*}
a _{\gamma ^*}(w,0)&=
2
\left|\frac{\partial U}{\partial x}(\gamma ^*(s_1))\frac{d\gamma 
^*}{ds}(s_1)\right|^2\ ,
\\
b_{\gamma ^*}(w,0) &= 
\nu (\gamma^*(s_1))
\left|
\frac{\partial U}{\partial x}(\gamma ^*(s_1))\frac{d\gamma ^*}{ds}(s_1)
\right|^2 
\,  .
\end{align*}
In this way, for any pair $(w_1,w_2)$ in 
\mbox{$\RR^{2n}\!\setminus\!\mathcal{D} $}
and any minimizing geodesic $\gamma ^*$ between $w_1$ and $w_2$,
\begin{list}{}{%
\parskip 0pt plus 0pt minus 0pt%
\topsep 0.5ex plus 0pt minus 0pt%
\parsep 0pt plus 0pt minus 0pt%
\partopsep 0pt plus 0pt minus 0pt%
\itemsep 0.5ex plus 0pt minus 0pt
\settowidth{\labelwidth}{--}%
\setlength{\labelsep}{0.5em}%
\setlength{\leftmargin}{\labelwidth}%
\addtolength{\leftmargin}{\labelsep}%
}
\item[--]
the function
$r\mapsto (a _{\gamma ^*}(w,r),b _{\gamma ^*}(w,r))$ is 
defined and $C^1$
\footnote{
This comes from this general result.
Let $f$ be a $C^2$ function defined on a neighborhood of $0$ in 
$\RR$, where it is $0$.
The function $\varphi $ defined as $\varphi (r)=\frac{f(r)}{r}$ 
if $r\neq 0$ and $\varphi (0)=f^\prime(0)$ is $C^1$.
\\
Indeed it is clearly $C^2$ everywhere except may be at $0$. Its first 
derivative is $\varphi ^\prime(r)=\frac{f(r)-rf^\prime(0)}{r^2}$. 
It is also continuous at $0$ since $\lim_{r\to 0}\varphi 
(r)=f^\prime(0)=\varphi (0)$. Its first derivative at $0$ exists if
$\lim_{r \to 0}\frac{\varphi (r)-\varphi (0)}{r}=\lim_{r\to 
0}\frac{f(r)-rf^\prime(0)}{r^2}$ exists. But this is the case, since 
$f$ being $C^2$, we have
\begin{eqnarray*}
\frac{f(r)-rf^\prime(0)}{r^2}&=&\frac{1}{r^2}\int_0^r 
[f^\prime(s)-f^\prime(0)]ds
\\
&=&\frac{1}{r^2}\int_0^r\int_0^s 
f^{\prime\prime}(t) dt ds
\\
&=&
\frac{1}{r^2}\int_0^r f^{\prime\prime}(t) [r-t]dt
\end{eqnarray*}
which leads to
$\varphi ^\prime(0)=\frac{1}{2}f^{\prime\prime}(0)$. We have also 
$$
\frac{f(r)-rf^\prime(r)}{r^2}=-\frac{1}{r^2}\int_0^r 
s f^{\prime\prime}(s) ds
$$
This implies
$$
\lim_{r\to 0} \varphi ^\prime(r) = \varphi^\prime(0)
$$
and therefore $\varphi ^\prime $ is continuous.
}
on $[0,d(w_1,w_2)]$,
\item[--]we have~:
\\[1em]$\displaystyle 
\frac{1}{d(w_1,w_2)}\,  \ddt \left.\left(\left.L(\Gamma(s,t))
\vrule height 0.51em depth 0.51em width 0em
\right|_{s_1}^{s_2}\right)\vrule height 0.51em depth 0.51em width 0em
\right|_{t=0}
\;+\; \lambda 
$\hfill \null \\[1em]\null \hfill $\displaystyle 
\leq \, 
b_{\gamma ^*}(w,d(w_1,w_2))
-
\kell (w)\,  a_{\gamma ^*}(w,d(w_1,w_2))
\  .
$\\
\end{list}
Also

\begin{lemma}
\label{lem1}
For any pair $(w_1,w_2)$ in \mbox{$\RR^{2n}\!\setminus\!\mathcal{D} $}
and any minimizing geodesic $\gamma ^*$ between $w_1$ and $w_2$,
$a_{\gamma ^*}(w,d(w_1,w_2))$ is non negative,
and if it is 
zero, the same holds for $b_{\gamma ^*}(w,d(w_1,w_2))$.
\end{lemma}

\begin{proof}
For any pair $w$ in $\RR^{2n}\!\setminus\!\mathcal{D} $,
the function $r\to a_{\gamma ^*}(w,r)$ is 
defined and continuous on $[0,d(w_1,w_2)]$.

If the real number
$$
a _{\gamma ^*}(w,0)\;=\;
2
\left|\frac{\partial U}{\partial x}(\gamma ^*(s_1))\frac{d\gamma 
^*}{ds}(s_1)\right|^2
$$
is zero, the level sets of $U$ being totally geodesic, we have
$$
\begin{array}{@{}c@{}}
\displaystyle 
\frac{\partial U}{\partial x}(\gamma ^*(r+s_1))
\frac{d\gamma ^*}{ds}(r+s_1) =  0
\  ,
\\[0.7em]
U(\gamma ^*(r+s_1))= U(w_1)
\  ,
\end{array}
\qquad \forall r \in [0,d(w_1,w_2)]
\  .
$$
and so $U(w_2)= U(w_1)$ and
$a_{\gamma ^*}(w,d(w_1,w_2))$ and $b_{\gamma 
^*}(w,d(w_1,w_2))$ are zero.

If instead that real number
is positive and 
\\[1em]$\displaystyle 
a _{\gamma ^*}(w,d(w_1,w_2))\;=\; \frac{U(\gamma 
^*(s_2))-U(w_1)}{d(w_1,w_2)}
\times
$\\\hfill \null \\\null \hfill $\displaystyle 
\times
\left[
\frac{\partial U}{\partial x}(\gamma ^*(s_2))
\frac{d\gamma ^*}{ds}(\gamma ^*(s_2))
+ 
\frac{\partial U}{\partial x}(w_1)\frac{d\gamma ^*}{ds}(s_1)
\right]
$\\[1em]
is non positive, then there exists $r_0$ in 
in $]0,d(w_1,w_2)]$ such that
\begin{list}{}{%
\parskip 0pt plus 0pt minus 0pt%
\topsep 0pt plus 0pt minus 0pt
\parsep 0pt plus 0pt minus 0pt%
\partopsep 0pt plus 0pt minus 0pt%
\itemsep 0pt plus 0pt minus 0pt
\settowidth{\labelwidth}{\\}%
\setlength{\labelsep}{0.5em}%
\setlength{\leftmargin}{\labelwidth}%
\addtolength{\leftmargin}{\labelsep}%
}
\item[\textit{either}]
$
U(\gamma ^*(r_0+s_1))= U(w_1)
$. But the level sets of $U$ being totally geodesic and $\gamma ^*$ 
being a minimizing geodesic between $w_1$ and $w_2$ and therefore
between
$w_1$ and $\gamma ^*(r_0+s_1)$, 
it follows from (the 
proof of)
\cite[Proposition A.3.2]{SanfelicePraly_TAC_12} that $\gamma ^*
$ takes its values in the level set $\{x\,  :\: U(x)=U(w_{1})\}$ at 
least on $[s_1,r_0+s_1]$. This implies
$$
\frac{\partial U}{\partial x}(\gamma ^*(s))\,  \frac{d\gamma 
^*}{ds}(s)\;=\; 0
$$
for all $s$ in $[s_1,r_0+s_1]$
and consequently in $[s_1,s_2]$. This yields $U(w_2)= U(w_1)$ and
$a_{\gamma ^*}(w,d(w_1,w_2))$ and $b_{\gamma 
^*}(w,d(w_1,w_2))$ are zero;
\item[\textit{or}] we have
\\[1em]$\displaystyle 
\frac{\partial U}{\partial x}(\gamma ^*(r_0+s_1))
\frac{d\gamma ^*}{ds}(\gamma ^*(r_0+s_1))
+ 
\frac{\partial U}{\partial x}(\gamma ^*(s_1))
\frac{d\gamma ^*}{ds}(s_1)
$\hfill \null \\[0.5em]\null \hfill $\displaystyle =\; 0
\  .
$\\[1em]
This implies that $\frac{\partial U}{\partial x}(\gamma ^*(s_1))
\frac{d\gamma ^*}{ds}(s_1)$ and
$\frac{\partial U}{\partial x}(\gamma ^*(r_0+s_1))
\frac{d\gamma ^*}{ds}(r_0+s_1)$ have opposite signs and so the function
$r\mapsto \frac{\partial U}{\partial x}(\gamma ^*(r+s_1))
\frac{d\gamma ^*}{ds}(r+s_1)$ must vanish on $]0,r_0[$. Again this 
implies $U(w_2)= U(w_1)$ and and $a_{\gamma ^*}(w,d(w_1,w_2))$ and $b_{\gamma ^*}(w,d(w_1,w_2))$ are zero.
\end{list}
\end{proof}

\par\vspace{1em}
Now, to each pair of integers $(i_1,i_2)$, we associate the compact set
\\[1em]$\displaystyle 
C_{i _1i_2}= \left\{
\vrule height 0.5em depth 0.5em width 0pt
(w_1,w_2)\in\RR^{2n}\,  :\:
\right.$\hfill \null \\\null \hfill $\displaystyle  \left.
\vrule height 0.5em depth 0.5em width 0pt
i_1\leq d(w_1,0)\leq i_1+1\; ,\  
i_2\leq d(w_2,0)\leq i_2+1
\right\}\ .
$
\par\vspace{1em}\noindent

\begin{lemma}
For any pair $(i_1,i_2)$, there exists a real number $\kell _{i_1i_2}$ 
such that, for all $(w_1,w_2)$ in 
\mbox{$C_{i_1i_2}\setminus\mathcal{D}$} and all $\kell $ larger or equal 
to $\kell_{i_1i_2}$, we have~:
\begin{equation}
\label{LP34}
b(w,d(w_1,w_2))-\kell \,  a(w,d(w_1,w_2)) \leq \frac{\lambda}{2}
\  .
\end{equation}
\end{lemma}

\begin{proof}
For the sake of getting a contradiction, assume there exist a pair $(i_1,i_2)$ and a 
sequence $(w_{1\indice },w_{2\indice },\gamma _{\indice }^*)_{\indice \in \NN}$ 
of points and minimizing geodesic satisfying
\begin{eqnarray}
\nonumber
d(w_{1\indice },w_{2\indice })&\neq & 0
\\[0.5em]\nonumber
 w_{1\indice }&=&\gamma _{\indice }^*(0)
\\[0.5em]\nonumber
w_{2\indice }&=& \gamma _{\indice }^*(d(w_{1\indice },w_{2\indice }))
\\\label{LP28}
&=& 
w_{1\indice }\;+\; 
\int_0^{d(w_{1\indice },w_{2\indice })} \frac{d\gamma _{\indice }^*}{ds}(s ) ds 
\end{eqnarray}
and
\\[1em]$\displaystyle 
-\frac{\lambda }{2}\;+\; b_{\gamma _{\indice }^*}(w_{1\indice },w_{2\indice },d(w_{1\indice },w_{2\indice }))
$\refstepcounter{equation}\label{LP29}\hfill$(\theequation)$
\\[0.7em]\null\hfill$\displaystyle
 \geq \; \indice \,  a_{\gamma _{\indice }^*}(w_{1\indice },w_{2\indice },d(w_{1\indice },w_{2\indice }))
\; \geq \; 0
\  .
$\\[1em]
Because the sequence
$(b_{\gamma _{\indice }^*}(w_{1\indice },w_{2\indice },d(w_{1\indice },w_{2\indice })))_{\indice \in \NN}$
is bounded, we 
have
\begin{equation}
\label{LP32}
\lim_{\indice \to \infty }
a_{\gamma _{\indice }^*}(w_{1\indice },w_{2\indice },d(w_{1\indice },w_{2\indice }))\;=\; 0
\  .
\end{equation}
The sequence $(w_{1\indice },w_{2\indice })_{\indice \in \NN}$ has a cluster point
$(w_{1\omega },w_{2\omega })$. To keep the notations simple, we still 
denote by $\indice $ the index of the subsequence for which we have 
convergence to this point.
\par\vspace{1em}\noindent
\underline{\textit{Case 1~: $w_{1\omega }=w_{2\omega }=w_{\omega }$.}}
Assume we have
$$
\lim_{\indice \to \infty } d(w_{1\indice },w_{2\indice })
=
\lim_{\indice \to \infty } d(w_{1\indice },w_{\omega })
=
\lim_{\indice \to \infty } d(w_{2\indice },w_{\omega })
=
0
\,  .
$$
By compactness, $C^1$ property and boundedness, 
there exists a real number 
$M$ and an integer $\indice _*$ such that, for all $\indice $ larger than $\indice _*$, we have
\\[1em]\null \quad $\displaystyle 
a_{\gamma _{\indice }^*}(w_{1\indice },w_{2\indice },d(w_{1\indice },w_{2\indice }))
$\hfill \null \\[0.5em]\null \hskip 4.5em$\displaystyle 
\begin{array}{@{}cl@{}}
\geq &\displaystyle 
a_{\gamma _{\indice }^*}(w_{1\indice },w_{2\indice },0)\;-\; M\,  d(w_{1\indice },w_{2\indice })
\\
\geq &\displaystyle 
2\left|
\frac{\partial U}{\partial x}(w_{1\indice })\frac{d\gamma _{\indice }^*}{ds}(0)
\right|^2 
\;-\; M\,  d(w_{1\indice },w_{2\indice })
\end{array}
$\\[1em]\null \quad $\displaystyle 
b_{\gamma _{\indice }^*}(w_{1\indice },w_{2\indice },dw_{1\indice },w_{2\indice }))
$\hfill \null \\[0.5em]\null \hskip 4.5em $\displaystyle 
\begin{array}{@{}cl@{}}
\leq &\displaystyle 
b_{\gamma _{\indice }^*}(w_{1\indice },w_{2\indice },0)\;+\; M\,  d(w_{1\indice },w_{2\indice })
\\
\leq &\displaystyle 
\nu (w_{1\indice })
\left|
\frac{\partial U}{\partial x}(w_{1\indice })\frac{d\gamma _{\indice }^*}{ds}(0)
\right|^2 
\;+\; M\,  d(w_{1\indice },w_{2\indice })
\end{array}
$\\[1em]
This implies
\\[1em]$\displaystyle 
-\frac{\lambda }{2\indice }+
\left[\frac{\nu (w_{1\indice })}{\indice }-2\right]
\left|
\frac{\partial U}{\partial x}(w_{1\indice })\frac{d\gamma _{\indice }^*}{ds}(0)
\right|^2 
$\hfill \null \\[0.7em]\null\hfill$
\geq \; -M\left[\frac{1}{\indice }+1\right]\,  d(w_{1\indice },w_{2\indice })
$\\[1em]
and therefore
\begin{equation}
\label{LP30}
\lim_{k\to \infty }\left|
\frac{\partial U}{\partial x}(w_{1\indice })\frac{d\gamma _{\indice }^*}{ds}(0)
\right|^2 \;=\; 0
\end{equation}

Also, by compactness $\frac{w_{2\indice }-w_{1\indice }}{d(w_{1\indice },w_{2\indice })}$ has a cluster point 
we denote $v_\omega $. With again $\indice $ as index for the subsequence (of the 
subsequence!), we have
$$
v_\omega \;=\; \lim_{\indice \to \infty }
\frac{w_{2\indice }-w_{1\indice }}{d(w_{1\indice },w_{2\indice })}
$$
But with (\ref{LP28}), this gives also
$$
v_\omega \;=\; \lim_{\indice \to \infty }
\frac{d\gamma _{\indice }^*}{ds}(0)
$$
which, with (\ref{LP30}), gives~:
$$
\frac{\partial U}{\partial x}(w_{\omega })\,  v_\omega \;=\; 0
$$
and implies~:
$$
\lim_{\indice \to \infty } b_{\gamma _{\indice }^*}(w_{1\indice },w_{2\indice },d(w_{1\indice },w_{2\indice }))
\;= \; 0
$$
Since $a_{\gamma _{\indice }^*}(w_{1\indice },w_{2\indice },d(w_{1\indice },w_{2\indice }))$ is non negative, 
this contradicts (\ref{LP29}).
\par\vspace{1em}\noindent
\underline{\textit{Case 2~: $w_{1\omega }\neq w_{2\omega }$.}}
Assume now we have
\begin{eqnarray*}
&\displaystyle 
\lim_{\indice \to \infty } d(w_{1\indice },w_{1\omega })
=
\lim_{\indice \to \infty } d(w_{2\indice },w_{2\omega })
=
0
\  ,
\\&\displaystyle 
d(w_{1\omega },w_{2\omega })
\neq 0
\  .
\end{eqnarray*}
$(w_{1\indice },w_{2\indice })$ is in the compact set
$C_{i_1i_2}$ and $\gamma _{\indice }^*$ is a minimal geodesic
at least on $[0,d(w_{1p},w_{2p})]$. So, from~:
\\[1em]$\displaystyle 
\sqrt{\underline{p}}\,  |w_1-w_2|\; \leq  \; d(w_1,w_2)\; \leq  \; 
\sqrt{\overline{p}}\,  |w_1-w_2|
$\hfill \null \\\null \hfill $\displaystyle  \forall (w_1,w_2)\in \mathcal{C}\times\mathcal{C}
\  ,
$\\[1em]
we get, for all $s$ in $[0,d(w_{1\indice },w_{2\indice })]$,
$$
\sqrt{\underline{p}}
\,  |\gamma _{\indice }^*(s)-w_{1\indice }|\; \leq \; d(\gamma _{\indice }^*(s),w_{1\indice })\; 
\leq \; d(w_{1\indice },w_{2\indice })
\; \leq  \;
D_{i_1i_2}
$$
where
$$
D_{i_1i_2}\;=\; \sup_{(w_1,w_2)\in C_{i_1i_2}} d(w_1,w_2)
$$
We have also
\begin{eqnarray*}
|w_{1\indice }|&\leq &
|w_{2\indice }-w_{1\indice }|\;+\; |w_{2\indice }|
\\
&\leq &
\frac{d(w_{2\indice },w_{1\indice })+d(w_{2\indice },0)}{\sqrt{\underline{p}}}
\  ,
\\&\leq &
\frac{D_{i_1i_2}+(i_2+1)}{\sqrt{\underline{p}}}
\  .
\end{eqnarray*}
With the completeness of the metric, this implies that $\gamma _{\indice }^*:[0,D_{i_1i_2}]\to \RR^n$
takes its values in a compact set
independent of the index $\indice $ and is a solution of the geodesic 
equation. It follows, from instance from \cite[Theorem 5, \S 1]{Filippov_Book_88}, 
that there exist a subsequence (of the subsequence !) again with 
$\indice $ as index and a solution $\gamma _{\omega }^*$ of the geodesic equation satisfying,
$$
\gamma _{\omega }^*(0) \;=\;  w_{1\omega}
\quad ,\qquad 
\gamma _{\omega }^* d(w_{1\omega },w_{2\omega })) \;=\;  w_{2\omega }
\  .
$$
and
$$%
\lim_{\indice \to \infty }\gamma _{\indice }^*(s)\;=\; \gamma _{\omega }^*(s)
\qquad \textrm{uniformly in} \  s \in [0,D_{i_1i_2}]
\  .
$$
Also each $\gamma _{\indice }^*$ being minimizing between $w_{1k}$ and 
$w_{2k}$, $\gamma _{\omega }^*$ is minimizing between $w_{1\omega }$ 
and $w_{2\omega }$ (see \cite[Lemma II|.4.2]{Sakai_Book_96}). 
With the definitions (\ref{LP31}) and (\ref{LP33}) of
$a_{\gamma ^*}(w_{1},w_{2},d(w_{1},w_{2}))$
and
$b_{\gamma ^*}(w_{1},w_{2},d(w_{1},w_{2}))$
and with (\ref{LP32}) we obtain
\\[1em]$\displaystyle 
a_{\gamma _{\omega  }^*}(w_{1\omega  },w_{2\omega  },d(w_{1\omega  },w_{2\omega  }))
$\hfill \null \\\null \hfill $
\begin{array}{@{}cl@{}}
=&\displaystyle 
\lim_{\indice \to \infty } a_{\gamma _{\indice }^*}(w_{1\indice },w_{2\indice },d(w_{1\indice },w_{2\indice }))
\  ,
\\
=& 0
\  ,
\end{array}
$\\[1em]$\displaystyle 
b_{\gamma _{\omega }^*}(w_{1\omega },w_{2\omega },d(w_{1\omega },w_{2\omega }))
$\hfill \null \\\null \hfill $\displaystyle 
=\quad  
\lim_{\indice \to \infty } b_{\gamma _{\indice }^*}(w_{1\indice },w_{2\indice },d(w_{1\indice },w_{2\indice }))
\  .
$\\[1em]
With Lemma \ref{lem1}, we get~:
$$
b_{\gamma _{\omega }^*}(w_{1\omega },w_{2\omega },d(w_{1\omega 
},w_{2\omega }))\;=\; 0
$$
So, as in the previous case, we get a contradiction of (\ref{LP29}).

To complete the proof of Proposition \ref{Prop_GlobSynch}, it is sufficient to observe that 
for any pair $(w_1,w_2)$ in $\RR^{2n}$, we can find $(i_1,i_2)$ such that 
$C_{i_1i_2}$ contains it. It is then sufficient to pick 
$\kell(w_1,w_2)$ larger or equal to $\kell_{i_1i_2}$ to obtain, with 
(\ref{LP34}),
$$ 
\mathfrak D^+ d(w_1,w_2) \; \leq \; 
  \ddt \left.\left(\left.L(\Gamma(s,t))
\vrule height 0.51em depth 0.51em width 0em
\right|_{s_1}^{s_2}\right)\vrule height 0.51em depth 0.51em width 0em
\right|_{t=0}
\; \leq \; -\frac{\lambda }{2}\,  d(w_1,w_2)
$$
This gives~:
\\[1em]$\displaystyle 
\sqrt{\underline{p}}\,  \left|W_1(w_1,w_2,t)-W_2(w_1,w_2,t)\right|
$\hfill \null \\[0.5em]\null \hfill $\displaystyle 
\begin{array}[b]{@{}cl@{}}
\leq &\displaystyle 
d\left(W_1(w_1,w_2,t)-W_2(w_1,w_2,t)\right)
\  ,
\\
\leq &\displaystyle 
\exp\left(-\frac{\lambda }{2}t\right)\,  d(w_1,w_2)
\  ,
\\[0.7em]
\leq &\displaystyle 
\exp\left(-\frac{\lambda }{2}t\right)\,  \sqrt{\overline{p}}\,  |w_1-w_2|
\  .
\end{array}$
\end{proof}

\bibliography{BibVA}

\end{document}